\documentclass[11pt]{article}

\usepackage[T1]{fontenc}
\usepackage[utf8]{inputenc}
\usepackage{xcolor}
\usepackage{calc}
\usepackage{amsmath}
\usepackage{amsthm}
\usepackage{amssymb}
\usepackage[all]{xy}

\usepackage{geometry}
\makeatletter
\theoremstyle{remark}
\newtheorem*{notation*}{\protect\notationname}
\theoremstyle{plain}
\newtheorem{thm}{\protect\theoremname}[section]
\theoremstyle{definition}
\newtheorem{defn}[thm]{\protect\definitionname}
\theoremstyle{definition}
\newtheorem{example}[thm]{\protect\examplename}
\theoremstyle{remark}
\newtheorem{rem}[thm]{\protect\remarkname}
\theoremstyle{plain}
\newtheorem{lem}[thm]{\protect\lemmaname}
\theoremstyle{plain}
\newtheorem{prop}[thm]{\protect\propositionname}
\theoremstyle{definition}
\newtheorem{xca}[thm]{\protect\exercisename}
\theoremstyle{remark}
\newtheorem*{claim*}{\protect\claimname}
\theoremstyle{plain}
\newtheorem{cor}[thm]{\protect\corollaryname}
\theoremstyle{definition}
\newtheorem*{xca*}{\protect\exercisename}
\theoremstyle{remark}
\newtheorem{note}[thm]{\protect\notename}

\usepackage[bookmarksnumbered]{hyperref}

 \theoremstyle{definition}

\makeatother

\usepackage{babel}
\providecommand{\claimname}{Claim}
\providecommand{\corollaryname}{Corollary}
\providecommand{\definitionname}{Definition}
\providecommand{\examplename}{Example}
\providecommand{\exercisename}{Exercise}
\providecommand{\lemmaname}{Lemma}
\providecommand{\notationname}{Notation}
\providecommand{\notename}{Note}
\providecommand{\propositionname}{Proposition}
\providecommand{\remarkname}{Remark}
\providecommand{\theoremname}{Theorem}

\global\long\def\om{\omega}%
\global\long\def\spec{\mathrm{Spec}}%
\global\long\def\ot{\otimes}%
\global\long\def\op{\oplus}%
\global\long\def\su{\subseteq}%
\global\long\def\A{\mathbb{A}}%
\global\long\def\F{\mathcal{F}}%
\global\long\def\C{\mathcal{C}}%
\global\long\def\Z{\mathbb{Z}}%
\global\long\def\Q{\mathbb{Q}}%
\global\long\def\P{\mathbb{P}}%
\global\long\def\proj{\mathrm{Proj}}%
\global\long\def\ua{^{\ast}}%
\global\long\def\da{_{\ast}}%
\global\long\def\G{\mathcal{G}}%
\global\long\def\H{\mathcal{H}}%
\global\long\def\hom{\mathrm{Hom}}%
\global\long\def\D{\mathcal{D}}%
\global\long\def\coh{\mathrm{Coh}}%
\global\long\def\qcoh{\mathrm{QCoh}}%
\global\long\def\tor{\mathrm{Tor}}%
\global\long\def\ord{\mathrm{ord}}%
\global\long\def\m{\mathfrak{m}}%
\global\long\def\ch{\mathrm{ch}}%
\global\long\def\td{\mathrm{td}}%
\global\long\def\E{\mathcal{E}}%

\begin{document}
\title{Grothendieck's proof of Hirzebruch-Riemann-Roch theorem}
\author{Giacomo Graziani}
\date{}
\maketitle
\begin{abstract}
The Riemann-Roch Theorem is one of the cornerstones of algebraic geometry,
connecting algebraic data (sheaf cohomology) with geometric ones (intersection
theory). This survey paper provides a self-contained introduction
and a complete proof of the Hirzebruch-Riemann-Roch (HRR) Theorem
for smooth projective varieties over an algebraically closed field.

Starting from the classical formulations for curves and surfaces,
we introduce the two modern tools necessary for the generalization:
the Grothendieck group $K_{0}(X)$ as the natural setting for the
Euler characteristic, and the Chow ring $A_{\bullet}(X)$ as the setting
for cycles and intersection theory. We then construct the fundamental
bridge between these two worlds\textemdash the Chern character ($\mathrm{ch}$)
and the Todd class ($\mathrm{td}$) \textemdash culminating in a full
proof of the HRR formula:

\[
\chi(X,\mathcal{E})=\int_{X}\mathrm{ch}(\mathcal{E})\cdot\mathrm{td}(X)
\]
We conclude by showing how this general formula recovers the classical
theorems for curves and surfaces.
\end{abstract}

\tableofcontents{}

\newpage{}

\section{Introduction}

Given an algebraic variety $X$, we can adopt (at least) two perspectives
in its study: the first one stems from the algebraic world, namely
Morita's theory, homological and commutative algebra which allow us
to identify $X$ with its category of coherent $\mathcal{O}_{X}$-modules.
On the other hand of course we have the geometry of $X$ which is
usually phrased in terms of property of particular embeddings (e.g.
in $\P^{n}$) or, and this is intrinsic in our variety, in terms of
its subvarieties and their mutual relations, that is, in its intersection
theory. The real interest in algebraic geometry comes from the interplay
between these points of view and the Hirzebruch-Riemann-Roch theorem
is one bright example of such a relation. 

The quest for this kind of relations has been ultimately started by
the work of Bernhard Riemann and Gustav Roch to address the problem
of computing the dimension of the space of meromorphic functions on
a compact Riemann surface with prescribed zeros and poles. On the
one side one finds that, given a divisor $D$\footnote{At this point, a divisor may be thought of as a compact way of describing
the zero (and polar) locus we're interested in. Namely here 
\[
D=\sum_{\mbox{finite}}n_{i}P_{i}
\]
is a formal sum where $P_{i}$ are points and $n_{i}\in\Z$. In this
case we simply define $\deg\left(D\right)=\sum n_{i}$.}, the relation is
\[
l\left(D\right)-l\left(K-D\right)=\deg\left(D\right)+1-g,
\]
where $l\left(D\right)$ is the dimension we were looking for, $K$
is a canonically defined divisor and $g$ is the genus (a topological
invariant). Phrased in more modern terms, one can associate to every
coherent $\mathcal{O}_{X}$-module $\F$ a sequence of finite dimensional vector
spaces $H^{i}\left(X,\F\right)$ and, under this perspective, the
integer $\ell\left(D\right)$ is just the dimension of $H^{0}\left(X,\mathcal{O}_{X}\left(D\right)\right)$.
Working harder one can show that 
\[
\ell\left(K-D\right)=\dim_{\Bbbk}H^{1}\left(X,\mathcal{O}_{X}\left(D\right)\right),
\]
this is a particular case of the Serre duality\footnote{For a proof of Riemann-Roch theorem in these terms, you can look at
\cite[Theorem IV.1.3, pag. 295]{Har} }. At this point one is lead to guess how a generalization to higher
dimensions looks like: the obvious generalization of the left hand
side is given by the (finite) sum
\[
\chi\left(X,\mathcal{O}_{X}\left(D\right)\right)=\sum_{i\ge0}\left(-1\right)^{i}\dim_{\Bbbk}H^{i}\left(X,\mathcal{O}_{X}\left(D\right)\right),
\]
this has the nice property of being additive on bounded exact sequences.
The tough part is to understand the right hand side. The two-dimensional
case, that is for smooth surfaces, was worked out essentially by Guido
Castelnuovo and can be written in cohomological terms as\footnote{For a proof, see \cite[Theorem V.1.6, pag. 362]{Har}.}
\[
\chi\left(S,\mathcal{O}_{S}\left(D\right)\right)=\frac{1}{2}D\cdot\left(D-K\right)+\chi\left(S,\mathcal{O}_{S}\right)
\]
where the $\cdot$ denotes the intersection product between divisors.
This can be seen as an evidence that the right hand side has to do
with the intersection theory of our variety, and this will be addressed
in Section \ref{sec:Intersection-theory}.

The main ideas are essentially two: in the first part we introduce
the Grothendieck group of an exact category, this will be a universal
group for ``functions'' which are additive on short exact sequences,
therefore it is the right source both for $\chi\left(X,\bullet\right)$
and for the so-called Chern character $\ch\left(\bullet\right)$, on the
other hand one can associate to $X$ a ring $A_{\bullet}\left(X\right)$
made up of classes of subvarieties modulo a suitable equivalence relation
and whose product is given essentially by intersecting such varieties.
One can wonder whether this can be further generalized, and the answer
is yes, but at the cost of another strong idea, which is ultimately
one of the cornerstones of Grothendieck's approach to geometry: the
focus of the theorem should be a morphism between two schemes rather
than a scheme itself. This leads to the general Grothendieck-Riemann-Roch
theorem and it is essentially the content of SGA6.

What we actually prove here is the so-called Hirzebruch-Riemann-Roch
theorem, although this name usually refers to the complex case, where
also transcendental methods are applied in the course of the proof. 
\begin{notation*}
Let us fix some notation and terminology: $\Bbbk$ will denote an algebraically
closed fields of arbitrary characteristic\footnote{Even if much of the work done here can be adapted with little effort
to arbitrary fields} and all geometric objects will be defined over $\Bbbk$. When we say
that $X$ is a variety, we mean that $X$ is a reduced and irreducible.
\end{notation*}
\newpage
\section{The $K$-theory}
\subsection*{Summary}

For the reader interested primarily in the proof of the Hirzebruch-Riemann-Roch Theorem, the technical construction of K-theory can be treated as a toolbox whose main properties we now outline.

To every smooth projective variety $X$, we associate an abelian group $K_0(X)$ generated by isomorphism classes of locally free sheaves (vector bundles). Since $X$ is smooth, any coherent sheaf admits a finite resolution by vector bundles; this fundamental result allows us to identify $K_0(X)$ with the group $G_0(X)$ generated by all coherent sheaves (Theorem \ref{thm:K0generatedbyprojective}). This group carries a commutative ring structure induced by the tensor product: $[\mathcal{F}] \cdot [\mathcal{G}] = [\mathcal{F} \otimes \mathcal{G}]$ (Remark \ref{rem:ringstructureprojectives}).

A crucial feature of the theory is its functoriality: any projective morphism $f: X \to Y$ induces a pushforward homomorphism $f_!: K_0(X) \to K_0(Y)$, defined via the alternating sum of higher direct images (Definition \ref{defn:fshriek}). This operation recovers classical invariants in specific cases: if $f$ is the map to a point, $f_![\mathcal{F}]$ is the Euler characteristic $\chi(X, \mathcal{F})$ (Example \ref{exa:eulercharacteristic}); if $i: Z \hookrightarrow X$ is a closed embedding, $i_!$ simplifies to the classical direct image $i_*$ (Proposition \ref{prop:fshriek}).

Finally, explicitly computing these groups is possible: for the projective space $\mathbb{P}^n$, the ring $K_0(\mathbb{P}^n)$ is $$K_0(\mathbb{P}^n)\simeq\frac{\Z [x]}{(1-x)^{n+1}}$$ where $x$ corresponds to $[\mathcal O _{\P ^n} (1)]$ (Theorem \ref{thm:K0projectivespace}).

\subsection{The Grothendieck group}

Here we introduce the basics about the Grothendieck group of rings
and schemes.
\begin{defn}
A category with exact sequences is a full additive subcategory $\mathcal{P}$
of an abelian category $\C$ which is closed under extensions, that
is, given an exact sequence 
\[
0\to P'\to P\to P''\to0
\]
in $\C$ with $P',P''$ in $\mathcal{P}$, then also $P$ is an object
of $\mathcal{P}$. An exact sequence in $\mathcal{P}$ is an exact
sequence in $\C$ involving only objects of $\mathcal{P}$.\footnote{One should also add some smallness condition $\mathcal{P}$: in \cite[Definition 3.1.1, pag.109]{Ros}
a category with exact sequences is supposed to have a small skeleton,
that is $\mathcal{P}$ is equivalent to a small category $\mathcal{P}_{0}$
(i.e. a category whose objects form a set).

We will ignore this kind of set theoretic issues and always suppose
that our categories have small skeleton: in fact, as long as we're
concerned with category of modules over a scheme or over a (possibly
non-commutative) ring, this condition is automatically satisfied (see
\cite[Examples 3.1.2 (2), (3), (7), pagg. 109-111]{Ros}, where he
also deals with other categories of modules.}
\end{defn}

\begin{example}
\label{exa:Our-main-examples}Our main examples of categories with
exact sequences will be 
\begin{itemize}
\item An abelian category $\C$ is a category with exact sequences as long
as it is equivalent to a small category.
\item For a ring $A$, let $\proj_{A}^{fg}$ be the category of finitely
generated projective $A$-modules and $\mathrm{Mod}_{A}^{fg}$ the
category of finitely generated $A$-modules: they are categories with
exact sequences, embedded in the category $\mathrm{Mod}_{A}$ of $A$-modules.
When $A$ is noetherian, the category $\mathrm{Mod}_{A}^{fg}$ is
itself abelian. Moreover $\proj_{A}^{fg}$ has the remarkable property
that every exact sequence is split.
\item For an algebraic variety $X$ (or, more generally, a scheme) the category
$\coh_{X}^{lf}$ of locally free coherent $\mathcal{O}_{X}$-modules, embedded
in the category $\qcoh_{X}$ of quasi-coherent $\mathcal{O}_{X}$-modules or
even in $\coh_{X}$ of coherent $\mathcal{O}_{X}$-modules when $X$ is locally
noetherian.
\end{itemize}
\end{example}

\begin{defn}
Fix a category with exact sequences $\mathcal{P}$ and let $S_{\mathcal{P}}$
be the set of isomorphism classes of objects of $\mathcal{P}$. The
Grothendieck group $K_{0}\left(\mathcal{P}\right)$ of $\mathcal{P}$
is defined as
\[
K_{0}\left(\mathcal{P}\right)=\frac{\Z\left[S_{\mathcal{P}}\right]}{R}
\]
where $\Z\left[S_{\mathcal{P}}\right]$ is the free abelian group
on $S_{\mathcal{P}}$ and $R$ is the subgroup generated by the relation
$\left[A\right]=\left[B\right]+\left[C\right]$ every time an exact
sequence of the form
\[
0\to B\to A\to C\to0
\]
exists in $\mathcal{P}$.
\end{defn}

\begin{example}
Let $A$ be a ring, we denote with $K_{0}\left(A\right)$ the group
$K_{0}\left(\mathrm{\proj}_{A}^{fg}\right)$ and by $G_{0}\left(A\right)$
the group $K_{0}\left(\mathrm{Mod}_{A}^{fg}\right)$. In the same
vein, given a projective variety $X$, we denote with $K_{0}\left(X\right)$
the group $K_{0}\left(\mathrm{Coh}_{X}^{lf}\right)$ and $G_{0}\left(X\right)$
the group $K_{0}\left(\coh_{X}\right)$ .
\end{example}

\begin{rem}
\label{rem:exactinducesmap}Let $\F:\C\to\D$ be an exact functor
between two categories with exact sequences, then one readily verifies
that we have a well defined group homomorphism
\begin{eqnarray*}
\F_{\ast}:K_{0}\left(\C\right) & \to & K_{0}\left(\D\right)\\
\left[X\right] & \mapsto & \left[\F\left(X\right)\right].
\end{eqnarray*}
This is a particular case of Example \ref{exa:Flowershriek} since an exact functor has vanishing higher derived functors ($R^i\mathcal{F}=0$ for $i>0$).
\end{rem}

\subsection{Dévissage theorem}
We see here how, in case the base variety or ring is regular, the $K$-groups we get considering coherent sheaves are the same that we get considering only locally free coherent sheaves.
\begin{lem}
\label{lem:alternatesumvanishes}Let $\mathcal{P}$ be a category
with exact sequences seen as a full subcategory of an abelian category
$\C$ and suppose that, for every epimorphism $\pi$ in $\mathcal{P}$,
$\ker\left(\pi\right)$ is still an object of $\mathcal{P}$. Let
\[
0\to X_{n}\to\dots\to X_{1}\to X_{0}\to0
\]
be an exact sequence in $\mathcal{P}$. Then the equality
\[
\sum_{i=0}^{n}\left(-1\right)^{i}\left[X_{i}\right]=0
\]
holds in $K_{0}\left(\mathcal{P}\right)$.
\end{lem}

\begin{proof}
Let us proceed by induction on $n$: the cases when $n\le2$ are clear.
Suppose the statement holds for $n-1$, then we can consider the two
exact sequences
\[
0\to K\to X_{1}\to X_{0}\to0
\]
\[
0\to X_{n}\to\dots\to X_{2}\to K\to0
\]
and apply the inductive hypothesis.
\end{proof}
\begin{lem}
\label{lem:K0classcomplexhomology}Let $\C$ be an abelian category
and let $X_{\bullet}$ be a bounded complex in $\C$. Then we have the
following equality in $K_{0}\left(\C\right)$
\[
\sum\left(-1\right)^{i}\left[X_{i}\right]=\sum\left(-1\right)^{i}\left[H_{i}\left(X_{\bullet}\right)\right].
\]
\end{lem}

\begin{proof}
Let $d_{i}:X_{i}\to X_{i-1}$ be the differentials, then we have exact
sequences
\[
0\to\mathrm{Im}\left(d_{i+1}\right)\to\ker\left(d_{i}\right)\to H_{i}\left(X_{\bullet}\right)\to0
\]
\[
0\to\ker\left(d_{i}\right)\to C_{i}\to\mathrm{Im}\left(d_{i}\right)\to0.
\]
We compute
\begin{eqnarray*}
\sum\left(-1\right)^{i}\left[H_{i}\left(X_{\bullet}\right)\right] & = & \sum\left(-1\right)^{i}\left[\ker\left(d_{i}\right)\right]-\sum\left(-1\right)^{i}\left[\mathrm{Im}\left(d_{i+1}\right)\right]\\
 & = & \sum\left(-1\right)^{i}\left[\ker\left(d_{i}\right)\right]+\sum\left(-1\right)^{i}\left[\mathrm{Im}\left(d_{i}\right)\right]\\
 & = & \sum\left(-1\right)^{i}\left[C_{i}\right].
\end{eqnarray*}
\end{proof}
\begin{thm}
\label{thm:K0generatedbyprojective}[Dévissage] Let $A$ be a regular ring\footnote{We say that a local ring $\left(A,\mathfrak{m},k\right)$ is a regular
local ring if it is noetherian and
\[
\dim A=\dim_{k}\frac{\mathfrak{m}}{\mathfrak{m}^{2}}
\]
We say that a (non necessarily local) ring $A$ is regular if it is
noetherian and for every prime ideal $\mathfrak{p}\in\spec\left(A\right)$,
the local ring $A_{\mathfrak{p}}$ is regular. }(resp. let $X$ be a smooth projective variety). Then the natural
inclusions $$K_{0}\left(A\right)\to G_{0}\left(A\right)\quad\quad\mbox{and}\quad\quad K_{0}\left(X\right)\to G_{0}\left(X\right)$$
are isomorphisms.\footnote{Look at the Resolution Theorem \cite[Theorem 3.1.13, pag. 121]{Ros}
(or the subsequent page for the $K_{1}$-analogue) for a generalization
of this result. In general one has a pletora of dévissage theorems
for reducing the study of a category to that of a well-behaved subcategory
by means of suitable exact sequences (e.g. resolutions).}
\end{thm}

\begin{proof}
Suppose $A$ is a regular ring and let $M$ be a finitely generated $A$-module, then, in view of Serre-Auslander-Buchsbaum's Theorem (see for example \cite[Theorem 8.62, page 491]{Rot}) there exists an exact sequence
\[
0\to P_{n}\to\dots\to P_{0}\to M\to0
\]
with $P_{i}$'s finitely generated and projective, hence the statement for $K_{0}\left(A\right)$ follows from Lemma \ref{lem:alternatesumvanishes}. For smooth projective varieties the proof follows the same pattern, except that the existence of a finite resolution made up of locally free coherent sheaves is a bit more delicate and follows from Serre-Auslander-Buchsbaum's Theorem combined with \cite[Exercise III.6.8, page 238]{Har}.
\end{proof}
From now on, when dealing with regular rings or varieties, we'll always identify the corresponding $K_{0}$ and $G_{0}$.
\begin{example}
\begin{enumerate}
\item Let $\Bbbk$ be a field, then two finitely generated $\Bbbk$-modules are isomorphic if and only if they have the same dimension. It follows that $K_{0}\left(\Bbbk\right)\simeq\Z$ induced by 
\[
V\mapsto\dim_{\Bbbk}V.
\]
\item More generally, let $A$ be a regular local ring, then all finitely generated projective modules are free and hence they have a well-defined rank. Like in the case of fields, and in view of Theorem \ref{thm:K0generatedbyprojective}, we deduce that $K_{0}\left(A\right)\simeq\Z$ induced by the rank function.
\end{enumerate}
\end{example}

\subsection{Multiplicative structure}

\begin{rem}\label{rem:ringstructureprojectives}
Let $R$ be any commutative ring. Then the assignment
\[
\left(\left[P\right],\left[Q\right]\right)\mapsto\left[P\ot_{R}Q\right]
\]
defines a commutative ring structure on $K_{0}\left(R\right)$. More generally if $X$ is a projective variety then $K_0 (X)$ becomes a commutative ring via
\[ [\F]\cdot [\G] = [\F\ot\G] .\]
If fact the definitions are well posed since locally free modules are flat, moreover associativity boils down to usual associativity for the tensor product.
\end{rem}

Diving deeper into homological algebra one could prove

\begin{prop}\label{prop:ringtor}
Let $R$ be a regular ring, then the assignment
\[
\left(\left[M\right],\left[N\right]\right)\mapsto\left[M\right]\cdot\left[N\right]=\sum_{i\ge0}\left(-1\right)^{i}\left[\tor_{i}^{R}\left(M,N\right)\right]
\]
makes $G_{0}\left(R\right)$ into a commutative ring such that the
inclusion $K_{0}\left(R\right)\to G_{0}\left(R\right)$ is a ring
isomorphism (cfr. Theorem \ref{thm:K0generatedbyprojective}
and Remark \ref{rem:ringstructureprojectives}).
\end{prop}

\begin{proof}
This is the formula obtained transporting the ring structure of $K_0 (R)$ via the Dévissage isomorphism of Theorem \ref{thm:K0generatedbyprojective}. See \cite[Chapter V]{Ser} for a discussion of intersection multiplicities via Tor.
\end{proof}
\begin{rem}
I leave as an exercise to the willing reader to show directly that
this definition really endows $G_{0}\left(R\right)$ with a ring structure.
The only non trivial point to show is associativity.
\end{rem}

\subsection{Functoriality}

\begin{example}
\label{exa:Flowershriek}Let $\F:\C\to\D$ be a left exact additive
functor between two abelian categories and suppose that its right
derived functors $\left\{ R^{i}\F\right\} _{i\ge0}$ exist and that,
for every object $X$, we have $R^{i}\F\left(X\right)=0$ for $i\gg0$
(depending on $X$). Then the assignment 
\[
\F_{!}:K_{0}\left(\C\right)\to K_{0}\left(\D\right)
\]
given by
\[
\F_{!}\left[X\right]=\sum_{i\ge0}\left(-1\right)^{i}\left[R^{i}\F\left(X\right)\right].
\]
is a well defined group homomorphism: if $0\to X\to Y\to Z\to0$
is an exact sequence in $\mathcal C$ then we have an exact sequence
\[
0\to \mathcal{F}(X) \to \mathcal F (Y) \to \mathcal F (Z) \to R^1 \mathcal{F}(X) \to R^1 \mathcal{F}(Y) \to R^1 \mathcal{F}(Z)\to\dots
\]
that is finite by assumption. Since $\mathcal D$ is abelian it has kernels hence Lemma \ref{lem:alternatesumvanishes} tells that $$\F_{!}\left[Y\right]=\F_{!}\left[X\right]+\F_{!}\left[Z\right]$$ 

Our main example will be the following: let $f:X\to Y$ be a morphism between projective varieties, in particular it has finite relative dimension and $R^i f_\ast$ preserves coherence so the discussion above applies to $f_\ast$  hence $(f_\ast ) _{!}:G_{0}\left(X\right)\to G_{0}\left(Y\right)$ is well defined and, when $Y$ is smooth, it restricts to $K_0 (X)\to K_0 (Y)$ in view of Theorem \ref{thm:K0generatedbyprojective}.
\end{example}

\begin{defn}\label{defn:fshriek}
    Let $f:X\to Y$ be a morphism between projective varieties, we denote
    \[f_! = (f_\ast ) _{!}:G_{0}\left(X\right)\to G_{0}\left(Y\right) \]
    the map introduced in Example \ref{exa:Flowershriek}. When $Y$ is smooth with an abuse of notation we also denote 
    \[f_! :K_{0}\left(X\right)\to K_{0}\left(Y\right) \]
    the induced map.
\end{defn}

\begin{xca}
Let $f:A\to B$ be a (commutative) ring homomorphism. We have an induced
homomorphism 
\begin{align*}
K_{0}\left(A\right) & \to K_{0}\left(B\right)\\
\left[P\right] & \mapsto\left[P\ot_{A}B\right].
\end{align*}
Show that this works as long as we're working with projective modules.
What conditions are needed to extend this map to $G_{0}\left(A\right)\to G_{0}\left(B\right)$?
\end{xca}

We give now some examples

\begin{example}\label{exa:eulercharacteristic}
    Let $X$ be a projective variety and consider the map $f:X\to \{ \ast\}$ to a point. Clearly the category of coherent sheaves on a point is just the category of finite dimensional vector spaces and they are classified by their dimension, hence
    \[K_0 (\{ \ast\}) = G_0 (\{ \ast\})=\mathbb Z .\]
    Since $R^i f_\ast \F = H^i \left( X,\F\right)$ we have $f_!\F=\sum_{i\ge0} (-1)^i\dim _k H^i \left( X,\F\right)$ that is $$f_!\F = \chi\left(X,\F\right)$$ 
    is the Euler characteristic of $\F$.
\end{example}

\begin{prop}\label{prop:fshriek}
    Let $f:X\to Y$ be a closed embedding between two projective varieties with $Y$ smooth. If $\F$ is a coherent sheaf on $X$ then
    \begin{enumerate}
        \item $f_![\F]=[f_\ast \F]$
        \item $\chi\left( Y, j_![\F]\right)=\chi\left(X,[\F]\right)$
    \end{enumerate}
\end{prop}

\begin{proof}If $\F$ is a coherent sheaf on $X$ then, since $R^i f_\ast \F$ is the sheaf associated with the presheaf
\[U\mapsto H^i \left( f^{-1} \left( U\right),\F_{\vert f^{-1} \left( U\right)}\right),\]
we see that $R^i f_\ast \F =0$ for $i\ge1$ and $f_\ast \F$ is just extension by 0 outside the image of $f$ (see \cite[Proposition III.8.1, pag. 250]{Har}). Then
\[ f_![\F]=[f_\ast \F] \]
for every coherent sheaf $\F$ on $X$. In particular we compute:
\[\chi\left( Y, j_![\F]\right)= \sum_{i\ge0}\left(-1\right)^{i}\chi\left(Y ,\left[R^{i}j_\ast \F\right]\right) = \chi\left(Y ,\left[j_\ast \F\right]\right) = \chi\left(X ,\left[ \F\right]\right) .\]
\end{proof}

\subsection{An important example: the Grothendieck group of projective spaces}
\begin{lem} \label{lem:chiO(a)}
Let $1\le a\le n$ be an integer and $\Bbbk$ be
a field, then 
$$\chi\left(\P_{\Bbbk}^{n},\mathcal{O}_{\P_{\Bbbk}^{n}}\left(-a\right)\right)=0 \quad\mbox{and}\quad \chi\left(\P_{\Bbbk}^{n},\mathcal{O}_{\P_{\Bbbk}^{n}}\left(a\right)\right)=\binom{n+a}{n} .$$
\end{lem}

\begin{proof}
In view of \cite[Theorem III.5.1, page 225]{Har} we have
\begin{eqnarray*}
\chi\left(\P_{\Bbbk}^{n},\mathcal{O}_{\P_{\Bbbk}^{n}}\left(-a\right)\right) & = & \dim_{\Bbbk}H^{0}\left(\P_{\Bbbk}^{n},\mathcal{O}_{\P_{\Bbbk}^{n}}\left(-a\right)\right)+\left(-1\right)^{n}\dim_{\Bbbk}H^{n}\left(\P_{\Bbbk}^{n},\mathcal{O}_{\P_{k}^{n}}\left(-a\right)\right)\\
 & = & \pm\dim_{\Bbbk}H^{n}\left(\P_{\Bbbk}^{n},\mathcal{O}_{\P_{\Bbbk}^{n}}\left(-a\right)\right).
\end{eqnarray*}
Since $\omega_{\P_{\Bbbk}^{n}}\simeq\mathcal{O}_{\P_{\Bbbk}^{n}}\left(-n-1\right)$,
by Serre duality\footnote{For us, the statement of Serre duality boils down to the following:
let $X$ be a smooth projective variety over a field $\Bbbk$ and let
$d=\dim X$. Denote with $\omega_{X}$ the canonical sheaf of $X$
(that is, the determinant of the sheaf of 1-forms) then, for every
coherent locally free $\mathcal{O}_{X}$-module $\F$ and for every $i\ge0$,
we have
\[
\dim_{\Bbbk}H^{i}\left(X,\F\right)=\dim_{\Bbbk}H^{d-i}\left(X,\F^{\vee}\ot\omega_{X}\right)
\]
where $\F^{\vee}=\mathcal{H}om_{O_{X}}\left(\F,\mathcal{O}_{X}\right)$ denotes
the dual sheaf. See \cite[Section III.7]{Har}, in particular \cite[Corollary III.7.7, page 244]{Har}.}we get
\[
\dim_{\Bbbk}H^{n}\left(\P_{\Bbbk}^{n},\mathcal{O}_{\P_{\Bbbk}^{n}}\left(-a\right)\right)=\dim_{\Bbbk}H^{0}\left(\P_{\Bbbk}^{n},\mathcal{O}_{\P_{\Bbbk}^{n}}\left(a-n-1\right)\right)=0
\]
as $a-n-1<0$. The same proof shows also the second statement.
\end{proof}
\begin{thm}
\label{thm:K0projectivespace}Let $\Bbbk$ be a field, then we have a
ring isomorphism
\[
K_{0}\left(\P_{\Bbbk}^{n}\right)\simeq\frac{\Z\left[x\right]}{\left(1-x\right)^{n+1}}
\]
with $x$ corresponding to the class of $\mathcal{O}_{\P_{\Bbbk}^{n}}\left(1\right)$.
\end{thm}

\begin{proof}
Let us set $X=\P_{\Bbbk}^{n}$ and let $\phi:\Z\left[x\right]\to K_{0}\left(\P_{\Bbbk}^{n}\right)$
be the map defined by 
\[
\phi\left(x^{t}\right)=\left[\mathcal{O}_{\P_{\Bbbk}^{n}}\left(t\right)\right].
\]
Note that it is a ring homomorphism. Let us see that $\phi$ is surjective:
in view of \cite[Corollary II.5.18, page 121]{Har} we see that $K_{0}\left(\P_{\Bbbk}^{n}\right)$
is generated by classes of line bundles $\left[\mathcal{O}_{\P_{\Bbbk}^{n}}\left(m\right)\right]$ for $m\in\Z$.
To prove surjectivity of $\phi$ we need to show that $\left[\mathcal{O}_{\P_{\Bbbk}^{n}}\left(-1\right)\right]$
lies in its image: consider now $\Bbbk\left[T_{0},\dots,T_{n}\right]$
as a graded ring: the sequence $T_{0},\dots,T_{n}$ is clearly regular
and hence, then, dualizing the Koszul complex associated with it,
we get an exact sequence
\[
0\to\mathcal{O}_{\P_{\Bbbk}^{n}}\to\bigoplus_{i=1}^{n+1}\mathcal{O}_{\P_{\Bbbk}^{n}}\left(1\right)\to\bigoplus_{i=1}^{{n+1 \choose 2}}\mathcal{O}_{\P_{\Bbbk}^{n}}\left(2\right)\to\dots\to\mathcal{O}_{\P_{\Bbbk}^{n}}\left(n+1\right)\to0.
\]
In view of Lemma \ref{lem:alternatesumvanishes}, denoting with $\xi=\left[\mathcal{O}_{\P_{\Bbbk}^{n}}\left(1\right)\right]$,
we have 
\begin{eqnarray*}
0 & = & 1-\left(n+1\right)\xi+{n+1 \choose 2}\xi^{2}+\dots+\left(-1\right)^{n+1}\xi^{n+1}\\
 & = & \left(1-\xi\right)^{n+1}
\end{eqnarray*}
which shows in turn that $\phi$ factors through 
\[
\tilde{\phi}:\frac{\Z\left[x\right]}{\left(1-x\right)^{n+1}}\to K_{0}\left(\P_{\Bbbk}^{n}\right)
\]
and also that 
\[
\left[\mathcal{O}_{\P_{\Bbbk}^{n}}\left(-1\right)\right]=\phi\left(\left(-1\right)^{n+1}\xi^{n}+\dots+{n+1 \choose 2}\xi-\left(n+1\right)\right).
\]
In particular $\tilde{\phi}$ is surjective and $K_{0}\left(\P_{\Bbbk}^{n}\right)$
is generated by the classes
\[
\left[\mathcal{O}_{\P_{\Bbbk}^{n}}\right],\left[\mathcal{O}_{\P_{\Bbbk}^{n}}\left(1\right)\right],\dots,\left[\mathcal{O}_{\P_{\Bbbk}^{n}}\left(n\right)\right],
\]
that is, by $1,\xi,\dots,\xi^{n}$. To conclude, we show that $1,\xi,\dots,\xi^{n}$ are linearly independent over $\mathbb{Z}$. Let
\[
\alpha=\sum_{i=0}^{n}a_{i}\xi^{i}=0
\]
with $\left(a_{0},\dots,a_{n}\right)\neq\left(0,\dots,0\right)$.
Set $j=\max\left\{ i\,\mid\,a_{i}\neq0\right\} $, then
\[
0=\alpha\cdot\xi^{-j}=a_{0}\left[\mathcal{O}_{\P_{\Bbbk}^{n}}\left(-j\right)\right]+\dots+a_{j}\left[\mathcal{O}_{\P_{\Bbbk}^{n}}\right],
\]
but then Example \ref{exa:eulercharacteristic} and Lemma \ref{lem:chiO(a)}
shows that
\[
0=\chi\left(\P_{\Bbbk}^{n},\alpha\cdot\xi^{-j}\right)=a_{j}\cdot\chi\left(\P_{\Bbbk}^{n},\mathcal{O}_{\P_{\Bbbk}^{n}}\right)=a_{j}
\]
which is a contradiction. 
\end{proof}
\newpage
\section{Intersection theory\label{sec:Intersection-theory}}

Throughout this section, unless otherwise specified, $X$ will denote a smooth projective variety over an algebraically closed field $\Bbbk$.  The smoothness assumption is fundamental as it allows us to identify Weil and Cartier divisors (enabling the intersection product defined in Section \ref{sub:intersection}) and endows the Chow group $A_\bullet(X)$ with a ring structure. The assumption on $\Bbbk$ simplifies the definition of the degree map (Remark \ref{rem:degree}).

\subsection{The Chow groups}
\begin{lem}
\label{lem:orderfunction}Let $A$ be a local noetherian 1-dimensional
domain with fraction field $K$. For any $a\in A\backslash\left\{ 0\right\} $,
set
\[
\ord_{A}\left(a\right)=\ell_{A}\left(\frac{A}{aA}\right).
\]
Then $\ord_{A}$ gives a well-defined group homomorphism $K^{\ast}\to\Z$.
\end{lem}

\begin{proof}
Clearly we have 
\[
\ell_{A}\left(\frac{A}{aA}\right)=\ell_{\frac{A}{aA}}\left(\frac{A}{aA}\right),
\]
hence we only need to see that $\frac{A}{aA}$ is an artinian ring.
Since it is noetherian, we only need to see that it is 0-dimensional,
but this is clear since it has only one prime ideal. Given $a,b\in A\backslash\left\{ 0\right\} $,
we have an exact sequence
\[
0\to\frac{A}{aA}\to\frac{A}{abA}\to\frac{A}{bA}\to0
\]
and hence
\[
\ord_{A}\left(ab\right)=\ord_{A}\left(a\right)+\ord_{A}\left(b\right)
\]
in view of the additivity of the length. It is just a formal fact
now that the assignment
\begin{eqnarray*}
K^{\ast} & \to & \Z\\
\frac{a}{b} & \mapsto & \ord_{A}\left(a\right)-\ord_{A}\left(b\right)
\end{eqnarray*}
is a group homomorphism. 
\end{proof}
\begin{defn}
Let $X$ be a projective variety, a $k$-cycle on $X$ is a finite
(formal) sum
\[
\sum n_{i}\left[V_{i}\right]
\]
where $V_{i}\su X$ are $k$-dimensional subvarieties and $n_{i}\in\Z$.
We denote the free abelian group on $k$-subvarieties of $X$ with
$Z_{k}\left(X\right)$. Let $W\su X$ be a $\left(k+1\right)$-dimensional
subvariety and let $f\in\Bbbk\left(W\right)^{\ast}$. Define 
\[
\left(f\right)=\sum_{V\su W}\ord_{V}\left(f\right)\cdot\left[V\right]
\]
where the sum is taken over codimension one subvarieties $V$ of $W$
and $\ord_{V}$ is the order function associated to the 1-dimensional
local ring $\mathcal{O}_{W,V}$\footnote{You may want to glance at \cite[Exercise I.3.13, page 22]{Har} }
(see Lemma \ref{lem:orderfunction}). We say that a $k$-cycle $\alpha$
is rationally equivalent to $0$, written $\alpha\sim0$, if
\[
\alpha=\sum\left(f_{i}\right)
\]
for some $f_{i}\in\Bbbk\left(W_{i}\right)^{\ast}$ and some $\left(k+1\right)$-subvarieties
$W_{i}\su X$. Finally, denote with $\mathrm{Rat_{k}\left(X\right)}$
the subgroup of $Z_{k}\left(X\right)$ generated by $k$-cycles rationally
equivalent to 0,
\[
A_{k}\left(X\right)=\frac{Z_{k}\left(X\right)}{\mathrm{Rat}_{k}\left(X\right)}
\]
and 
\[
A_{\bullet}\left(X\right)=\bigoplus_{k\ge0}A_{k}\left(X\right).
\]
\end{defn}

\begin{example}
If $d=\dim X$, then clearly we have $A_{i}\left(X\right)=0$ for
$i>d$ and, as $X$ is irreducible, we have 
\[
A_{d}\left(X\right)=\Z\cdot\left[X\right]\simeq\Z.
\]
\end{example}

\begin{rem}
One important fact in the above construction is that, given a morphism
$f:X\to Y$ between projective $\Bbbk$-varieties, given a subvariety
$V\su X$, we define
\[
f_{\ast}\left[V\right]=\begin{cases}
\deg\left(V\to f\left(V\right)\right)\cdot\left[f\left(V\right)\right] & \mbox{if \ensuremath{\dim V=\dim f\left(V\right)}}\\
0 & \mbox{otherwise}
\end{cases},
\]
then $f_{\ast}$ gives a graded group homomorphism 
\[
f_{\ast}:A_{\bullet}\left(X\right)\to A_{\bullet}\left(Y\right),
\]
i.e. it respects rational equivalence (see \cite[Theorem 1.4, page 11]{Ful}). 
\end{rem}

\begin{lem}
\label{lem:pushforwardfunctorial}Let $X\overset{g}{\longrightarrow}Y\overset{f}{\longrightarrow}Z$
be morphisms between projective smooth varieties. Then 
\[
\left(f\circ g\right)_{\ast}=f\da\circ g\da
\]
as maps $A_{\bullet}\left(X\right)\to A_{\bullet}\left(Z\right)$.
\end{lem}

\begin{proof}
Working case by case, one is ultimately lead to showing the following:
\begin{claim*}
Let $X$ be irreducible and let $f,g$ be finite morphisms, then $$\deg\left(f\circ g\right)=\deg\left(f\right)\circ\deg\left(g\right)$$
\end{claim*}
which is equivalent to the multiplicativity of the degree in finite field extensions.
\end{proof}
\begin{rem}
\label{rem:degree}Consider now the structure morphism $X\to\left\{ \ast\right\} $
to a point, hence we have $f_{\ast}\alpha=0$ for every $\alpha\in A_{k}\left(X\right)$
for $k>0$ and\footnote{Here we used the fact that $\Bbbk$ is algebraically closed, the definition for general $\Bbbk$ is
\[
\sum n_{i}\left[P_{i}\right]\mapsto\sum n_{i}\cdot\left[\kappa\left(P_{i}\right):\Bbbk\right].
\]
to take the degree of a point into account.
}
\begin{eqnarray*}
f_{\ast}:A_{0}\left(X\right) & \to & A_{0}\left(\left\{ \ast\right\} \right)=\Z\\
\sum n_{i}\left[P_{i}\right] & \mapsto & \sum n_{i}
\end{eqnarray*}
The map $\pi_{\ast}$ is usually denoted with 
\[
\int_{X}:A_{\bullet}\left(X\right)\to\Z.
\]
\end{rem}

\subsection{Intersection with divisors}\label{sub:intersection}

\begin{defn}\label{defn:intersectionwithsubvariety}
Let $X$ be a projective variety and $j:Y\to X$ be the embedding
map of a $k$-dimensional irreducible subvariety. For a Cartier divisor
$D\su X$, set
\[
D\cdot Y:=j_{\ast}\left[j^{\ast}\mathcal{O}_{X}\left(D\right)\right]\in A_{k-1}\left(X\right)
\]
where $\left[j\ua\mathcal{O}_{X}\left(D\right)\right]$ is the Weil divisor
on $Y$ associated to the line bundle $j^{\ast}\mathcal{O}_{X}\left(D\right)$.
\end{defn}

\begin{rem}
\label{rem:pairingdivisors} In view of the smoothness assumption the pairing in Definition \ref{defn:intersectionwithsubvariety} induces
\[
\mathrm{Pic}\left(X\right)\ot_{\Z}Z_{k}\left(X\right)\to A_{k-1}\left(X\right).
\]
\end{rem}

\begin{prop}
Let $X$ be a $d$-dimensional projective variety, then
\begin{enumerate}
\item given two Cartier divisors $D,E\su X$, denoting with $\left[D\right],\left[E\right]$
the corresponding Weil divisors, we have
\[
D\cdot\left[E\right]=E\cdot\left[D\right]
\]
in $A_{d-2}\left(X\right)$;
\item the pairing of Remark \ref{rem:pairingdivisors} induces a map
\[
\mathrm{Pic}\left(X\right)\ot_{\Z}A_{k}\left(X\right)\to A_{k-1}\left(X\right).
\]
\end{enumerate}
\end{prop}

\begin{proof}
The proof of point $1)$ is quite involved, so we refer to \cite[Theorem 2.4, pag. 35]{Ful}
for a detailed discussion. Once one believes that $1)$ is true, then
$2)$ comes easily: given a $\left(k+1\right)$-dimensional subvariety
$V\su X$ and a rational function $f$ on $V$, then by part $1)$
\[
D\cdot\left[\left(f\right)\right]=\left(f\right)\cdot\left[D\right]=0
\]
since $\left(f\right)=0$ in $\mathrm{Pic}\left(V\right)$.
\end{proof}

\subsection{Projective bundles}
\begin{defn}
Let $X$ be a projective variety and let $\mathcal{E}$ be a locally
free coherent $\mathcal{O}_{X}$-module of rank $r+1$. We define $\P_{X}\left(\mathcal{E}\right)$
as the $\P^{r}$-fibration over $X$ with fibres defined as $\P_{X}\left(\mathcal{E}\right)_{x}=\P\left(\mathcal{E}_{x}\right)$
for every $x\in X$. We call $\P_{X}\left(\mathcal{E}\right)$ the
projective bundle associated to $\mathcal{E}$. We call the natural
map $\P_{X}\left(\mathcal{E}\right)\to X$ its structure map.
\end{defn}

\begin{rem}\label{rem:bundleconstruction}
There are many other ways to introduce $\P_{X}\left(\mathcal{E}\right)$,
here we review two of them:
\begin{description}
\item [{Scheme-theoretic}] Since $\mathcal{E}$ is coherent, we have
a quasi-coherent sheaf of $\mathcal{O}_{X}$-algebras $\mathrm{Sym}\left(\mathcal{E}\right)$
which is locally isomorphic to $\mathcal{O}_{X}\left[T_{0},\dots,T_{r}\right]$,
therefore we can define $\P_{X}\left(\mathcal{E}\right)$ as the relative
projective space
\[
\pi:\P_{X}\left(\mathcal{E}\right)=\underline{\proj}_{X}\left(\mathrm{Sym}\left(\mathcal{E}\right)\right)\to X.
\]
In view of the previous remark, every point of $X$ has an affine
open neighborhood 
\[
U=\spec\left(A\right)\su X
\]
such that
\[
\P_{X}\left(\mathcal{E}\right)\times_{X}U\simeq\P_{A}^{r},
\]
whence the name. Moreover the various $\mathcal{O}_{\P_{A}^{r}}\left(1\right)$
glue together to give a canonical invertible sheaf $\mathcal{O}_{\P_{X}\left(\mathcal{E}\right)}\left(1\right)$
on $\P_{X}\left(\mathcal{E}\right)$. See \cite[Section II.7, pag. 160-162]{Har}
for a more precise account and a universal property. Let us only point
out that, in view of \cite[Proposition II.7.10, pag. 161]{Har}, $\pi$
is a projective morphism whenever $X$ admits an ample line bundle,
therefore, if $X$ is a projective variety, also $\P_{X}\left(\mathcal{E}\right)$
is.
\item [{Bundle-theoretic}] We can see $\mathcal{E}$ as a vector bundle
on $p:E\to X$, hence we have a covering $\left\{ U_{i}\right\} _{i}$
of $X$ and isomorphisms $\psi_{i}:\pi^{-1}\left(U_{i}\right)\to U_{i}\times\A^{r+1}$
such that, for every $i\neq j$, the maps
\[
\psi_{i}\circ\psi_{j}^{-1}:\left(U_{i}\cap U_{j}\right)\times\A^{r+1}\to\left(U_{i}\cap U_{j}\right)\times\A^{r+1}
\]
are given by
\[
\left(P,x\right)\mapsto\left(P,M_{ij}\cdot x\right)
\]
for $M_{ij}\in\mathrm{GL}_{r+1}\left(\mathcal{O}_{X}\left(U_{i}\cap U_{j}\right)\right)$.
To define $\P_{X}\left(\mathcal{E}\right)$, we glue the patches 
\[
U_{i}\times\P^{r}\overset{p_{1}}{\longrightarrow}U_{i}
\]
along
\begin{align*}
\left(U_{i}\cap U_{j}\right)\times\P^{r} & \to\left(U_{i}\cap U_{j}\right)\times\P^{r}\\
\left(P,x\right) & \mapsto\left(P,\overline{M_{ij}}\cdot x\right)
\end{align*}
where $\overline{M_{ij}}\in\mathrm{PGL}_{r+1}\left(\mathcal{O}_{X}\left(U_{i}\cap U_{j}\right)\right)$
denotes the class of $M_{ij}$.
\end{description}
\end{rem}

\begin{rem}
\label{rem:canonicalsheafprojectivebundle}On $\P_{X}\left(\mathcal{E}\right)$
there is another ``distinguished'' sheaf, namely the pull-back $\pi^{\ast}\mathcal{E}$:
it is locally free of rank $r+1$ (since $\mathcal{E}$ is) and a
choice of a local basis amounts to the choice of local coordinates
for $\P_{X}\left(\mathcal{E}\right)$, hence we have a natural map
$\pi^{\ast}\mathcal{E}\to\mathcal{O}_{\P_{X}\left(\mathcal{E}\right)}\left(1\right)$
which is surjective as $\mathcal{O}_{\P_{X}\left(\mathcal{E}\right)}\left(1\right)$
is generated by its sections.
\end{rem}

\begin{defn}
Let $X$ be a projective variety, $\mathcal{E}$ a locally free coherent
$\mathcal{O}_{X}$-module of rank $r+1$ and $\pi:\P_{X}\left(\mathcal{E}\right)\to X$
the associated projective bundle. Denote with $D_{\mathcal{E}}\su\P_{X}\left(\mathcal{E}\right)$
the Cartier divisor corresponding to $\mathcal{O}_{\P_{X}\left(\mathcal{E}\right)}\left(1\right)$.
For every $i\ge-r$, set
\begin{align*}
s_{i}\left(\mathcal{E}\right):A_{k}\left(X\right) & \to A_{k-i}\left(X\right)\\
\alpha & \mapsto\pi\da\left(D^{r+i}_{\mathcal{E}}\cdot\pi\ua\alpha\right),
\end{align*}
called the $i$-th Segre class.
\end{defn}

\begin{prop}
\label{prop:propertiessegreclasses}Let $X$ be a projective variety:
\begin{enumerate}
\item For every $\mathcal{E}$ locally free coherent $\mathcal{O}_{X}$-module of
rank $r+1$, we have $s_{i}\left(\E\right)=0$ for $i<0$ and $s_{0}\left(\E\right)=\mathrm{id}$;
\item For every $\mathcal{E},\F$ locally free coherent $\mathcal{O}_{X}$-modules,
we have 
\[
s_{i}\left(\E\right)\circ s_{j}\left(\F\right)=s_{j}\left(\F\right)\circ s_{i}\left(\E\right)
\]
for all $i,j\ge-r$;
\item If $Y$ is a projective variety and $f:Y\to X$ is a morphism, then,
for every $\mathcal{E}$ locally free coherent $\mathcal{O}_{X}$-module and
every cycle $\alpha$ on $X$, then Projection Formula holds
\[
f\da\left(s_{i}\left(f\ua\E\right)\left(\alpha\right)\right)=s_{i}\left(\E\right)\left(f\da\alpha\right);
\]
\item For every line bundle $\mathcal{L}$ on $X$ and $i\ge0$, we have
\begin{align*}
s_{i}\left(\mathcal{L}\right):A_{k}\left(X\right) & \to A_{k-i}\left(X\right)\\
\alpha & \mapsto\left(-1\right)^{i}D^i\cdot\alpha.
\end{align*}
\end{enumerate}
\end{prop}

\begin{proof}
See \cite[Proposition 3.1, pag. 48]{Ful}.
\end{proof}
\begin{cor}
\label{cor:surjectiveandsplitinjective}Let $X$ be a projective variety,
$\mathcal{E}$ a locally free coherent $\mathcal{O}_{X}$-module of rank $r+1$
and $\pi:\P_{X}\left(\mathcal{E}\right)\to X$ the associated projective
bundle. Then
\begin{enumerate}
\item $\pi_{\ast}:A_{\bullet}\left(\P_{X}\left(\E\right)\right)\to A_{\bullet}\left(X\right)$
is surjective;
\item $\pi^{\ast}:A_{\bullet}\left(X\right)\to A_{\bullet}\left(\P_{X}\left(\E\right)\right)$
is split injective.
\end{enumerate}
\end{cor}

\begin{proof}
It follows from Proposition \ref{prop:propertiessegreclasses} that
for all $\alpha\in A_{\bullet}\left(X\right)$
\[
\alpha=\pi\da\left(rD_{\E}\cdot\pi\ua\alpha\right),
\]
hence $\pi\da$ is surjective. Let 
\begin{align*}
F:A_{k+r}\left(\P_{X}\left(\E\right)\right) & \to A_{k}\left(X\right)\\
\beta & \mapsto\pi\da\left(rD_{\E}\cdot\beta\right),
\end{align*}
then $F\circ\pi\ua=\mathrm{id}$.
\end{proof}
\begin{thm}[Splitting construction]
\label{thm:splittingconstruction}Let $X$ be a projective variety
and let $\mathcal{E}$ be a locally free coherent $\mathcal{O}_{X}$-module
of rank $r+1$. Then there exists a projective variety $Y$ and a
morphism $f:Y\to X$ such that
\begin{enumerate}
\item $f$ decomposes as a sequence of projective bundles;
\item the map $f\da:A_{\bullet}\left(Y\right)\to A_{\bullet}\left(X\right)$ is surjective;
\item the map $f\ua:A_{\bullet}\left(X\right)\to A_{\bullet}\left(Y\right)$ is split
injective;
\item there exists a sequence
\[
0=\mathcal{E}_{0}\su\mathcal{E}_{1}\su\dots\su\mathcal{E}_{r}=f^{\ast}\mathcal{E}
\]
of locally free $\mathcal{O}_{Y}$-modules such that $\dfrac{\mathcal{E}_{i}}{\mathcal{E}_{i-1}}$
is locally free of rank 1 for every $i=1,\dots,r$.
\end{enumerate}
\end{thm}

\begin{proof}
We work by induction on $r$. If $r=0$, then there is nothing to
prove. Suppose $r>0$ and set 
\[
Y'=\P_{X}\left(\mathcal{E}^{\lor}\right)\overset{f'}{\longrightarrow}X.
\]
By the above, we have an exact sequence of locally free sheaves\footnote{Exercise: why is $\tilde{\mathcal{E}}$ also locally free?}
\[
0\to\tilde{\mathcal{E}}\to\left(f'\right)\ua\mathcal{E}^{\lor}\to\mathcal{O}_{Y'}\left(1\right)\to0,
\]
moreover, since $\mathrm{rk}\left(\mathcal{O}_{Y'}\left(1\right)\right)=1$, we have
$\mathrm{rk}\left(\tilde{\mathcal{E}}\right)=\mathrm{rk}\left(\left(f'\right)\ua\mathcal{E}\right)-1=\mathrm{rk}\left(\mathcal{E}\right)-1=r$
we conclude by induction and Corollary \ref{cor:surjectiveandsplitinjective}.
\end{proof}

\subsection{Chern and Todd classes}
\begin{defn}
Let $X$ be a projective variety and let $\mathcal{E}$ be a locally
free coherent $\mathcal{O}_{X}$-module of rank $r+1$. Consider the Segre
classes of $\E$ as endomorphisms of $A_{\bullet}\left(X\right)$, we define
recursively
\begin{align*}
c_{0}\left(\E\right) & =\mathrm{id}\\
c_{1}\left(\E\right) & =-s_{1}\left(\E\right)\\
c_{2}\left(\E\right) & =-s_{2}\left(\E\right)+s_{1}\left(\E\right)^{2}\\
c_{n}\left(\E\right) & =-\sum_{i=1}^{n}s_{i}\left(\E\right)c_{n-i}\left(\E\right),
\end{align*}
called the Chern classes of $\E$\footnote{You may wonder why such a complicated recursive definition. One could define the total Segre class $$s\left( \E\right)=\sum_{i\ge0} s_i \left( \E\right)$$ that turns out to be an isomorphism. Its inverse is the total Chern class and it's given by $$c \left( \E\right) = \sum_{i\ge0} c_i \left( \E\right)$$ of Remark \ref{rem:chernclassesfinite}.}.
\end{defn}

\begin{rem}
\label{rem:chernclassesfinite}We see that $c_{i}\left(\E\right)$
gives a map $A_{k}\left(X\right)\to A_{k-i}\left(X\right)$, therefore
we have the feeling that the sequence of $c_{i}\left(\E\right)$ is
bounded: it is true (namely, one can show that $c_{i}\left(\E\right)=0$
for $i>\mathrm{rk}\left(\E\right)$, see \cite[Theorem 3.2, pag. 50]{Ful}). Therefore
the total Chern class
\[
c\left(\E\right)=\sum_{i\ge0}c_{i}\left(\E\right)
\]
is well defined.
\end{rem}

The following properties are just a formal translation of the corresponding
ones in Proposition \ref{prop:propertiessegreclasses}
\begin{prop}
\label{prop:propertieschernclasses}Let $X$ be a projective variety:
\begin{enumerate}
\item For every $\mathcal{E},\F$ locally free coherent $\mathcal{O}_{X}$-modules,
we have 
\[
c_{i}\left(\E\right)\circ c_{j}\left(\F\right)=c_{j}\left(\F\right)\circ c_{i}\left(\E\right)
\]
for all $i,j$;
\item If $Y$ is a projective variety and $f:Y\to X$ is a morphism, then,
for every $\mathcal{E}$ locally free coherent $\mathcal{O}_{X}$-module and
every cycle $\alpha$ on $X$, then Projection Formula holds
\[
f\da\left(c_{i}\left(f\ua\E\right)\left(\alpha\right)\right)=c_{i}\left(\E\right)\left(f\da\alpha\right).
\]
\end{enumerate}
\end{prop}

\begin{example}
\label{exa:chernclassinvertible}Let $\mathcal{L}$ be an invertible
sheaf on $X$, then $\mathcal{L}\simeq\mathcal{O}_{X}\left(D\right)$
for some divisor $D$. In view of Remark \ref{rem:chernclassesfinite},
we see that $c_{i}\left(\mathcal{L}\right)=0$ for $i\ge2$, therefore
\begin{align*}
c\left(\mathcal{L}\right) & =1+c_{1}\left(\mathcal{L}\right)\\
 & =1-s_{1}\left(\mathcal{L}\right)
\end{align*}
where $s_{1}\left(\mathcal{L}\right)$ is just the intersection with
$D$.
\end{example}

\begin{notation*}
By abuse of notation, whenever there exists a divisor $D$ such that
\begin{align*}
c_{i}\left(\E\right):A_{\bullet}\left(X\right) & \to A_{\bullet}\left(X\right)\\
\alpha & \mapsto D\cdot\alpha,
\end{align*}
we'll write $c_{i}\left(\E\right)=D$. With this notation, in view
of Example \ref{exa:chernclassinvertible}, $c\left(\mathcal{O}_{X}\left(D\right)\right)=1+D$
for any divisor $D$.
\end{notation*}
The main property of the total Chern class is
\begin{prop}[Whitney sum]
\label{prop:whitneysum}Let 
\[
0\to\F\to\G\to\mathcal{H}\to0
\]
be an exact sequence of locally free coherent $\mathcal{O}_{X}$-modules on
a smooth projective variety $X$. Then 
\[
c\left(\G\right)=c\left(\F\right)\circ c\left(\H\right).
\]
\end{prop}

\begin{proof}
It is \cite[Theorem 3.2 (e), pag. 51]{Ful} for $t=1$.
\end{proof}
\begin{example}
\label{exa:tangentbundlePn} Let us compute $c\left(\mathrm{T}_{\P^{n}}\right)$,
where $\mathrm{T}$ denotes the tangent sheaf. We have the Euler
exact sequence
\[
0\to\mathcal{O}_{\P^{n}}\to\mathcal{O}_{\P^{n}}\left(1\right)^{\op n+1}\to\mathrm{T}_{\P^{n}}\to0
\]
and hence, in view of Proposition \ref{prop:whitneysum}, we have
\[
c\left(\mathcal{O}_{\P^{n}}\left(1\right)^{\op n+1}\right)=c\left(\mathcal{O}_{\P^{n}}\right)\circ c\left(\mathrm{T}_{\P^{n}}\right),
\]
with $c\left(\mathcal{O}_{\P^{n}}\right)=1$ and 
\begin{align*}
c\left(\mathcal{O}_{\P^{n}}\left(1\right)^{\op n+1}\right) & =\left(1+H\right)^{n+1}.
\end{align*}
Then
\[
c_{k}\left(\mathrm{T}_{\P^{n}}\right)=\binom{n+1}{k}\cdot H^{k},
\]
where $H^{k}$ is the class of a $k$-dimensional linear subspace\footnote{Here we're cheating a bit: we haven't proved that all $k$-dimensional
linear subspaces of $\P^{n}$ are equivalent. Take it as an act of
faith or see \cite[Example 1.9.3, pag. 23]{Ful}.}.
\end{example}

\begin{rem}
\label{rem:splittingtrick}The way one uses Proposition \ref{prop:whitneysum}
and Theorem \ref{thm:splittingconstruction} is as follows: in view
of Example \ref{exa:chernclassinvertible}, given a locally free coherent
$\mathcal{O}_{X}$-module $\E$ of rank $r+1$ on a smooth projective variety
$X$, we have a product
\[
c\left(\E\right)=\prod_{i=1}^{r+1}\left(1+\alpha_{i}\right)
\]
on some projective bundle over $X$. In view of Corollary \ref{cor:surjectiveandsplitinjective},
we see that the $k$-th symmetric polynomial in the $\alpha_{i}$'s
(treating them as formal variables) is actually $c_{k}\left(\E\right)$,
hence every symmetric polynomial in the $\alpha_{i}$'s in expressible
in terms of the Chern classes. For a (more) formal explanation of
this principle one can see \cite[Remark 3.2.3, pag. 54]{Ful}. We also point out that, using this trick, given a morphism $f:X\to Y$
between projective varieties and a locally free coherent $\mathcal{O}_{Y}$-module
$\E$, we have
\[
c_{k}\left(f^{\ast}\E\right)=f^{\ast}c_{k}\left(\E\right).
\]
A proof of this can be found in \cite[Section 5.4]{EH}.
\end{rem}

\begin{example}
\label{exa:formulacheernroots}Using the splitting principle one can
see how Chern roots behave under linear algebra operations: if $\E$
has Chern roots $\alpha_{1},\dots,\alpha_{n}$, then
\begin{itemize}
\item $\E^{\lor}$ has Chern roots $-\alpha_{1},\dots,-\alpha_{n}$;
\item $S^{k}\E$ (the symmetric power) has Chern roots $\alpha_{i_{1}}+\dots+\alpha_{i_{k}}$
for $1\le i_{1}\le\dots\le i_{k}\le n$;
\item $\bigwedge^{k}\E$ (the exterior power) has Chern roots $\alpha_{i_{1}}+\dots+\alpha_{i_{k}}$
for $1\le i_{1}<\dots<i_{k}\le n$.
\end{itemize}
In particular $\det\left(\E\right)$ has Chern root $\alpha _1 + \dots +\alpha_n$. For a proof of this formulas see \cite[Remark 3.2.3, pag. 54]{Ful}
\end{example}

\begin{defn}\label{defn:ChernToddcharacters}
The formal variables $\alpha_{1},\dots,\alpha_{r+1}$ introduced in
Remark \ref{rem:splittingtrick} are called the Chern roots of $\E$.
We set\footnote{With this notation we mean 
\begin{align*}
\exp\left(x\right) & =1+x+\frac{1}{2}x^{2}+\dots\\
\frac{x}{1-\exp\left(-x\right)} & =1+\frac{1}{2}x+\frac{1}{12}x^{2}+\dots
\end{align*}
as formal power series in $\Q\left[\left[x\right]\right]$.}
\begin{align*}
\ch\left(\E\right) & =\sum_{i=1}^{r+1}\exp\left(\alpha_{i}\right)\\
\td\left(\E\right) & =\prod_{i=1}^{r+1}\frac{\alpha_{i}}{1-\exp\left(-\alpha_{i}\right)}
\end{align*}
as endomorphisms of $A_{\bullet}\left(X\right)\ot\Q$. They are called
the Chern and Todd character, respectively.
\end{defn}

\begin{rem}
The Chern and Todd characters are actually given by finite sums, as
follows from Remark \ref{rem:chernclassesfinite}. Moreover they are
clearly symmetric in the $\alpha_{i}$'s, so that they are expressible
using the Chern characters.
\end{rem}

\begin{prop}
\label{prop:cherntoddadditive}Let
\[
0\to\F\to\G\to\mathcal{H}\to0
\]
be an exact sequence of locally free coherent $\mathcal{O}_{X}$-modules on
a smooth projective variety $X$. Then 
\begin{align*}
\ch\left(\G\right) & =\ch\left(\F\right)+\ch\left(\H\right)\\
\td\left(\G\right) & =\td\left(\F\right)\circ\td\left(\H\right).
\end{align*}
\end{prop}

\begin{proof}
Applying the splitting construction three times, we find a morphism
$f:Y\to X$ from a projective variety $Y$ over which $\F,\G$ and
$\H$ have a filtration like in Theorem \ref{thm:splittingconstruction}.
There are (at least) two good reasons for which our starting exact
sequence stays exact when pulled-back to $Y$: first note that $f$
is constructed as a sequence of projective bundles, therefore, locally,
it looks like
\[
\P_{A}^{r}\to\spec\left(A\right)
\]
and hence it is flat. Otherwise one notes that we're working with
locally free modules, therefore our sequence is locally split. Since
the functor $f\ua$ is additive, it preserve the exactness of locally
split exact sequences.\footnote{Recall that, when working with $\mathcal{O}_{Y}$-modules, exactness is checked
locally.} In view of Proposition \ref{prop:whitneysum}, we deduce that the
set of Chern roots of $\G$ is just the union of those of $\F$ and
those of $\H$, therefore the Proposition reduces to a formal rearrangement
in the definitions of $\ch\left(\G\right)$ and $\td\left(\G\right)$.
\end{proof}
\noindent\fcolorbox{black}{white}{\begin{minipage}[t]{1\columnwidth - 2\fboxsep - 2\fboxrule}%
\[
\mbox{Black Box: turning the tables}
\]
So far we have defined Chern classes as endomorphisms of $A_{\bullet}\left(X\right)$. To have a more geometric flavour, we shall (and need to) see them as actual elements of $A_{\bullet}\left(X\right)$, and we can address this as follows: for a smooth projective variety $X$ one would like to put a commutative ring structure on $A^{\bullet}\left(X\right)$ taking, for $D,E$ irreducible cycles of codimension $d$ and $e$ respectively, \[ \left[D\right]\cdot\left[E\right]=\sum n_{i}\left[W_{i}\right] \] where the $W_{i}$'s are the irreducible components of $D\cap E$ and the $n_{i}$'s are integers defined in some smart way, called intersection multiplicities. This turns out to be possible and paves the way to a very deep subject, namely the so-called intersection theory. Anyway, in our specific case (that is, smooth projective varieties over an algebraically closed field) the theory goes without any great harm. One may want to page through \cite[Chapters 6,7]{Ful} for the whole geometric picture, while in the book \cite{Ser} an algebraic definition of intersection multiplicity, reminiscent of the ring structure introduced in Proposition \ref{prop:ringtor}, is given and generalized, leading to Serre's Intersection Multiplicity Conjectures, which is at present one of the most outstanding open problem in Commutative and Homological Algebra (see also Section 8 in \cite{Hoc} for a discussion of the problem and the state of the art).

Since $A_{\bullet}\left(X\right)$ and $A^{\bullet}\left(X\right)$ are the same abelian group, this gives a ring structure on $A_{\bullet}\left(X\right)$ where $\left[X\right]$ is the unit element. We will identify $c_i \left(\E\right)$ with $$c_i \left(\E\right)\cdot\left[X\right]\in A_{\bullet}\left(X\right),$$ then all the statements and properties we have seen in this section carry over as statements about cycles.%
\end{minipage}}

\begin{rem}\label{rem:degreebundle}
Now that we view Chern classes as elements of the Chow ring $A_\bullet(X)$, we can define the degree of a locally free sheaf $\E$ as
\[
\deg \left(\E\right) = \deg\left( c_1 \left(\E\right)\right).
\]
Using the splitting principle (Remark \ref{rem:splittingtrick} and Example \ref{exa:formulacheernroots}), one can verify that $$c_1(\E) = c_1(\det \E).$$ Thus, this definition coincides with the classical definition of the degree as the degree of the determinant line bundle.
\end{rem}

\begin{rem}
In view of Proposition \ref{prop:cherntoddadditive}, the previous
trick allows us to define the Chern character $\ch\left(\bullet\right)$
as a group homomorphism
\[
\ch\left(\bullet\right):K_{0}\left(X\right)\to A_{\bullet}\left(X\right)\ot\Q.
\]
Moreover, to a smooth projective variety $X$, we can associate an
element
\[
\td\left(X\right)=\td\left(\mathrm{T}_{X}\right)\in A_{\bullet}\left(X\right)\ot\Q
\]
called the Todd class of $X$, where $\mathrm{T}_{X}$ is the tangent
sheaf of $X$.
\end{rem}

\begin{example}
\label{exa:toddtangentbundlePn}
In view of the multiplicativity of the Todd class (Proposition \ref{prop:cherntoddadditive}) applied to the Euler exact sequence (from Example \ref{exa:tangentbundlePn}) we get:
\[
\mathrm{td}(\mathcal{O}_{\P^{n}}\left(1\right)^{\op n+1}) = \mathrm{td}(\mathcal{O}_{\P^{n}}) \cdot \mathrm{td}(\mathrm{T}_{\P^{n}})
\]
The Todd class of the trivial bundle $\mathrm{td}(\mathcal{O}_{\P^{n}})$ is $1$. For the sum, we use multiplicativity (Proposition \ref{prop:cherntoddadditive}) and the definition of the Todd class for a line bundle (Definition \ref{defn:ChernToddcharacters}) whose first Chern class is $H$:
\[
\mathrm{td}(\mathcal{O}_{\P^{n}}\left(1\right)^{\op n+1}) = \mathrm{td}(\mathcal{O}_{\P^{n}}(1))^{n+1} = \left( \frac{H}{1-\exp(-H)} \right)^{n+1}
\]
Substituting this back, we obtain the formula:
\[
\mathrm{td}\left(\mathrm{T}_{\P^{n}}\right) = \frac{H^{n+1}}{\left(1-\exp\left(-H\right)\right)^{n+1}},
\]
which, as usual, has to be meant as a Taylor expansion in $A_{\bullet}\left(\P^{n}\right)\ot\Q$.
\end{example}
\newpage
\section{Hirzebruch-Riemann-Roch for smooth projective varieties}

The aim of this final section is to state and prove the Hirzebruch-Riemann-Roch
theorem for smooth projective varieties. This is stated as
\begin{thm}
Let $X$ be a smooth projective variety, then the diagram
\[
\xymatrix{K_{0}\left(X\right)\ar[d]_{\chi\left(X,\bullet\right)}\ar[rr]^{\ch\left(\bullet\right)\cdot\td\left(X\right)} &  & A_{\bullet}\left(X\right)\ot\Q\ar[d]^{\int_{X}}\\
\Z\ar@{^{(}->}[rr] &  & \Q
}
\]
is commutative, that is, for every locally free coherent $\mathcal{O}_{X}$-module
$\E$, we have
\[
\chi\left(X,\E\right)=\int_{X}\ch\left(\E\right)\cdot\td\left(X\right).
\]
\end{thm}

\subsection{The case of $\protect\P^{n}$}
\begin{thm}
The diagram
\[
\xymatrix{K_{0}\left(\P^{n}\right)\ar[d]_{\chi\left(\P^{n},\bullet\right)}\ar[rr]^{\ch\left(\bullet\right)\cdot\td\left(\P^{n}\right)} &  & A_{\bullet}\left(\P^{n}\right)\ot\Q\ar[d]^{\int_{\P^{n}}}\\
\Z\ar@{^{(}->}[rr] &  & \Q
}
\]
is commutative.
\end{thm}

\begin{proof}
Let $\F$ be a locally free coherent $\mathcal{O}_{\P^{n}}$-module, we want
to show that
\[
\chi\left(\P^{n},\F\right)=\int_{\P^{n}}\ch\left(\F\right)\cdot\td\left(\P^{n}\right).
\]
In view of Theorem \ref{thm:K0projectivespace} and the fact that
$\ch$ is a group homomorphism, it is enough to show it for $\F=\mathcal{O}_{\P^{n}}\left(d\right)$
with $d=0,\dots,n$. On the one hand we know from Lemma \ref{lem:chiO(a)}
that 
\[
\chi\left(\P^{n},\mathcal{O}_{\P^{n}}\left(d\right)\right)=\binom{n+d}{n},
\]
on the other hand $\ch\left(\mathcal{O}_{\P^{n}}\left(d\right)\right)=\exp\left(dH\right)$
where $H\su\P^{n}$ is a hyperplane. In view of Example \ref{exa:toddtangentbundlePn}
we have
\[
\td\left(\P^{n}\right)=\frac{H^{n+1}}{\left(1-\exp\left(-H\right)\right)^{n+1}},
\]
and hence\footnote{To show the equality
\[
\int_{\P^{n}}\ch\left(\mathcal{O}_{\P^{n}}\left(d\right)\right)\cdot\td\left(\P^{n}\right)=\binom{n+d}{n}
\]
one uses the following trick: the point is to compute the $n$-th
coefficient in the Taylor expansion of the function $f\left(x\right)=e^{dx}\dfrac{x^{n+1}}{\left(1-e^{-x}\right)^{n+1}}$, but this is exactly the residue $\mathrm{res}_{0}\left(\dfrac{e^{dx}}{\left(1-e^{-x}\right)^{n+1}}\right)$,
which can be computed by the change of variables $y=1-e^{-x}$.}
\[
\int_{\P^{n}}\ch\left(\mathcal{O}_{\P^{n}}\left(d\right)\right)\cdot\td\left(\P^{n}\right)=\binom{n+d}{n}=\chi\left(\P^{n},\mathcal{O}_{\P^{n}}\left(d\right)\right).
\]
\end{proof}

\subsection{The case of smooth projective varieties}
\begin{lem}
\label{lem:enoughclosedembedding}Let $X$ be a smooth projective
variety and let $j:X\to\P^{n}$ be a closed embedding. If the diagram
\[
\xymatrix{K_{0}\left(X\right)\ar[rr]^{\ch\left(\bullet\right)\cdot\td\left(X\right)}\ar[d]_{j_{!}} &  & A_{\bullet}\left(X\right)\ot\Q\ar[d]^{j\da\ot\Q}\\
K_{0}\left(\P^{n}\right)\ar[rr]_{\ch\left(\bullet\right)\cdot\td\left(\P^{n}\right)} &  & A_{\bullet}\left(\P^{n}\right)\ot\Q
}
\]
is commutative, then for every locally free coherent $\mathcal{O}_{X}$-module
$\E$, we have
\[
\chi\left(X,\E\right)=\int_{X}\ch\left(\E\right)\cdot\td\left(X\right).
\]
\end{lem}

\begin{proof}
We see, in view of Proposition \ref{prop:fshriek}, that the composition
\[
\xymatrix{K_{0}\left(X\right)\ar[r]^{j_{!}} & K_{0}\left(\P^{n}\right)\ar[rr]^{\chi\left(\P^{n},\bullet\right)} &  & \Z}
\]
is just $\chi\left(X,\bullet\right)$. We are then reduced to showing that,
for $\alpha\in A_{\bullet}\left(X\right)$, we have
\[
\int_{Y}j_{\ast}\left(\alpha\right)=\int_{X}\alpha,
\]
but, in view of the very definition of $\int$, this is Lemma \ref{lem:pushforwardfunctorial}.
\end{proof}
\begin{lem}
\label{lem:enoughtoverifythis}Let $j:Y\to X$ be a closed embedding
between smooth projective varieties. If, for every locally free coherent
$\mathcal{O}_{Y}$-module $\E$, there exists a locally free coherent resolution
\[
0\to\F_{r}\to\dots\to\F_{1}\to\F_{0}\to j_{\ast}\E\to0
\]
by $\mathcal{O}_{X}$-modules such that
\[
\sum_{k=0}^{r}\left(-1\right)^{k}\ch\left(\F_{k}\right)=j\da\left(\frac{\ch\left(\E\right)}{\td\left(\mathcal{N}_{X|Y}\right)}\right),
\]
then the diagram
\[
\xymatrix{K_{0}\left(Y\right)\ar[rr]^{\ch\left(\bullet\right)\cdot\td\left(Y\right)}\ar[d]_{j_{!}} &  & A_{\bullet}\left(Y\right)\ot\Q\ar[d]^{j_{\ast}\ot\Q}\\
K_{0}\left(X\right)\ar[rr]_{\ch\left(\bullet\right)\cdot\td\left(X\right)} &  & A_{\bullet}\left(X\right)\ot\Q
}
\]
is commutative.
\end{lem}

\begin{proof}
We have an exact sequence
\[
0\to\mathrm{T}_{Y}\to j^{\ast}\mathrm{T}_{X}\to\mathcal{N}_{Y|X}\to0
\]
of locally free $\mathcal{O}_{Y}$-modules, therefore, in view of Proposition
\ref{prop:cherntoddadditive} and Remark \ref{rem:splittingtrick}
we see that
\begin{equation}
j^{\ast}\td\left(X\right)=\td\left(j^{\ast}\mathrm{T}_{X}\right)=\td\left(Y\right)\cdot\td\left(\mathcal{N}_{Y|X}\right).\label{eq:exactsequencetangent-1}
\end{equation}
We need to see that, for $y\in K_{0}\left(Y\right)$, the equality
\[
j\da\left(\ch\left(y\right)\cdot\td\left(Y\right)\right)=\ch\left(j_{!}\left(y\right)\right)\cdot\td\left(X\right)
\]
holds. In view of Theorem \ref{thm:K0generatedbyprojective},
we only need to check it when $y$ is the class of a locally free
coherent $\mathcal{O}_{Y}$-module $\E$. Let
\[
0\to\F_{r}\to\dots\to\F_{1}\to\F_{0}\to j_{\ast}\E\to0
\]
be a resolution as in the statement, then we have
\[
j_{!}\left[\E\right]=\sum_{k=0}^{r}\left(-1\right)^{k}\left[\F_{k}\right].
\]
In view of Proposition \ref{prop:cherntoddadditive} we have
\begin{align*}
\ch\left(j_{!}\left[\E\right]\right)\cdot\td\left(X\right) & =\ch\left(\sum_{k=0}^{r}\left(-1\right)^{k}\left[\F_{k}\right]\right)\cdot\td\left(X\right)\\
 & =\sum_{k=0}^{r}\left(-1\right)^{k}\ch\left(\F_{k}\right)\cdot\td\left(X\right)\\
 & =j\da\left(\ch\left(\E\right)\cdot\td\left(\mathcal{N}_{X|Y}\right)^{-1}\right)\cdot\td\left(X\right)\\
 & \overset{\ast}{=}j\da\left(\ch\left(\E\right)\cdot\td\left(\mathcal{N}_{X|Y}\right)^{-1}\cdot j^{\ast}\td\left(X\right)\right)\\
 & =j\da\left(\ch\left(\E\right)\cdot\td\left(Y\right)\right)
\end{align*}
where the equality marked with $\ast$ follows from Projection Formula
in Proposition \ref{prop:propertieschernclasses} and the last one
is Formula (\ref{eq:exactsequencetangent-1}).
\end{proof}
We need an intermediate step
\begin{lem}
\label{lem:specialcaseprojectivebundle}Let $X$ be a smooth projective
variety and let $\E$ be a locally free coherent sheaf of rank $r$,
set $Y=\P_{X}\left(\E\op\mathcal{O}_{X}\right)$. Then the diagram
\[
\xymatrix{K_{0}\left(X\right)\ar[rr]^{\ch\left(\bullet\right)\cdot\td\left(X\right)}\ar[d]_{i_{!}} &  & A_{\bullet}\left(X\right)\ot\Q\ar[d]^{i\da\ot\Q}\\
K_{0}\left(Y\right)\ar[rr]_{\ch\left(\bullet\right)\cdot\td\left(Y\right)} &  & A_{\bullet}\left(Y\right)\ot\Q
}
\]
is commutative, where the map $i:X\to Y$ is the one induced by the
inclusion $\mathcal{O}_{X}\su\E\op\mathcal{O}_{X}$. 
\end{lem}

\begin{proof}
Note that $i$ is a section for the structure map $\pi:Y\to X$, moreover
from Remark \ref{rem:canonicalsheafprojectivebundle} we have an exact
sequence of locally free sheaves
\[
0\to\mathcal{L}\to\pi\ua\left(\E\op\mathcal{O}_{X}\right)\to\mathcal{O}_{Y}\left(1\right)\to0
\]
which, dualized, gives\footnote{\begin{xca*}
Show that $\pi\ua\left(\E^{\lor}\right)\simeq\left(\pi\ua\E\right)^{\lor}$.
\end{xca*}
}
\[
0\to\mathcal{O}_{Y}\left(-1\right)\to\pi\ua\left(\E^{\lor}\op\mathcal{O}_{X}\right)\to\mathcal{L}^{\lor}\to0.
\]
Note that $\pi\ua\left(\E\op\mathcal{O}_{X}\right)$ has a canonical section
\[
\mathcal{O}_{Y}=\pi\ua\mathcal{O}_{X}\to\pi\ua\left(\E^{\lor}\op\mathcal{O}_{X}\right).
\]
Note that, composing it with the map $\pi\ua\left(\E^{\lor}\op\mathcal{O}_{X}\right)\to\mathcal{L}^{\lor}$,
we get a section $s^{\lor}:\mathcal{O}_{Y}\to\mathcal{L}^{\lor}$ whose zero
locus is really $i\left(X\right)\simeq X$. We also get an exact sequence
on $Y\backslash X$ (see Proposition \ref{prop:zerolocusimage})
\[
0\to\bigwedge^{r}\mathcal{L}\to\dots\to\bigwedge^{2}\mathcal{L}\to\mathcal{L}\overset{s}{\longrightarrow}\mathcal{O}_{Y}\to\mathcal{O}_{Z\left(s\right)}\to0.
\]
 Let $\F$ be a locally free coherent $\mathcal{O}_{X}$-module, then 
\begin{align*}
\pi\ua\F\ot\mathcal{O}_{Z\left(s\right)} & \simeq\pi\ua\F\ot i\da\mathcal{O}_{X}\\
 & \simeq i\da\left(i\ua\pi\ua\F\ot\mathcal{O}_{X}\right)\\
 & \simeq i\da\F
\end{align*}
and hence we get a locally free resolution on $Y$
\[
0\to\pi\ua\F\ot\bigwedge^{r}\mathcal{L}\to\dots\to\pi\ua\F\ot\bigwedge^{2}\mathcal{L}\to\pi\ua\F\ot\mathcal{L}\overset{s}{\longrightarrow}\pi\ua\F\to i\da\F\to0.
\]
In order to apply Lemma \ref{lem:enoughtoverifythis} we need to see
that
\[
\sum_{k=0}^{r}\left(-1\right)^{k}\ch\left(\pi\ua\F\ot\bigwedge^{k}\mathcal{L}\right)=i\da\left(\frac{\ch\left(\F\right)}{\td\left(\mathcal{N}_{Y|X}\right)}\right).
\]
First we show that $\mathcal{N}_{Y|X}\simeq i\ua\mathcal{L}^{\lor}$:
we have an exact sequence 
\[
0\to\mathcal{I}\to\mathcal{O}_{Y}\to\mathcal{O}_{X}\to0
\]
on $Y$, giving that 
\[
\mathrm{Tor}_{1}^{\mathcal{O}_{Y}}\left(\mathcal{O}_{X},\mathcal{O}_{X}\right)\simeq\frac{\mathcal{I}}{\mathcal{I}^{2}}=\mathcal{N}_{Y|X}^{\lor},
\]
on the other hand, in view of Remark \ref{rem:koszultor}, we have
\[
i\ua\mathcal{L}=\mathrm{Tor}_{1}^{\mathcal{O}_{Y}}\left(\mathcal{O}_{X},\mathcal{O}_{X}\right)
\]
and hence $\mathcal{N}_{Y|X}\simeq i\ua\mathcal{L}^{\lor}$. One can
show that
\[
i\da\left(\frac{\ch\left(\F\right)}{\td\left(i\ua\mathcal{L}^{\lor}\right)}\right)=\frac{\ch\left(\pi\ua\F\right)}{\td\left(\mathcal{L}^{\lor}\right)}\cdot c_{r}\left(\mathcal{L}^{\lor}\right)
\]
 (cf. \cite[Section 14.1, pag. 244]{Ful}), so that we need to show
the equality 
\[
\sum_{k=0}^{r}\left(-1\right)^{k}\ch\left(\bigwedge^{k}\mathcal{L}\right)=\frac{c_{r}\left(\mathcal{L}^{\lor}\right)}{\td\left(\mathcal{L}^{\lor}\right)},
\]
but, if $a_{1},\dots,a_{r}$ are the Chern roots of $\mathcal{L}$,
then, from Example \ref{exa:formulacheernroots}, 
\begin{align*}
\sum_{k=0}^{r}\left(-1\right)^{k}\ch\left(\bigwedge^{k}\mathcal{L}\right) & =\sum_{k=0}^{r}\left(-1\right)^{k}\left(\sum_{1\le i_{1}<\dots<i_{k}\le r}\exp\left(a_{i_{1}}+\dots+a_{i_{k}}\right)\right)\\
 & =\prod_{i=1}^{r}\left(1-\exp\left(a_{i}\right)\right)\\
 & =a_{1}\cdot\dots\cdot a_{r}\cdot\prod_{i=1}^{r}\frac{1-\exp\left(a_{i}\right)}{a_{i}}\\
 & =c_{r}\left(\mathcal{L}^{\lor}\right)\cdot\td\left(\mathcal{L}^{\lor}\right)^{-1}.
\end{align*}
\end{proof}
\begin{thm}[Hirzebruch-Riemann-Roch]\label{thm:HRR}
Let $X$ be a smooth projective variety, then the diagram
\[
\xymatrix{K_{0}\left(X\right)\ar[d]_{\chi\left(X,\bullet\right)}\ar[rr]^{\ch\left(\bullet\right)\cdot\td\left(X\right)} &  & A_{\bullet}\left(X\right)\ot\Q\ar[d]^{\int_{X}}\\
\Z\ar@{^{(}->}[rr] &  & \Q
}
\]
is commutative, that is, for every locally free coherent $\mathcal{O}_{X}$-module
$\E$, we have
\[
\chi\left(X,\E\right)=\int_{X}\ch\left(\E\right)\cdot\td\left(X\right).
\]
\end{thm}

\begin{proof}
Let $j:X\to\P^{n}$ be a closed embedding, $\E$ be a locally free
coherent $\mathcal{O}_{X}$-module. Fix a point $0\in\P^{1}$ and consider
the blowing up
\[
\pi:M=\mathrm{Bl}_{X\times\left\{ 0\right\} }\left(\P^{n}\times\P^{1}\right)\to\P^{n}\times\P^{1}
\]
together with the projection map $q:M\to\P^{1}$. For $P\neq0$ we
have $q^{-1}\left(P\right)\simeq\P^{n}$, while $q^{-1}\left(0\right)$
is the union of the exception divisor $E\simeq\P_{X}\left(\mathcal{N}_{X|\P^{n}}\op\mathcal{O}_{X}\right)$
and $Y\simeq\mathrm{Bl}_{X}\left(\P^{n}\right)$.

Let $i_E: E \to M$, $i_Y: Y \to M$ be the inclusions of the components of the fiber over $0$, and let $i_\infty: \mathbb{P}^n \to M$ be the inclusion of the fiber over $\infty$.

Note that the product
$j\times\P^{1}:X\times\P^{1}\to\P^{n}\times\P^{1}$ makes the preimage
of $X\times\left\{ 0\right\} $ locally principal, hence there exists
a unique $v:X\times\P^{1}\to M$ making the diagram
\[
\xymatrix{X\times\P^{1}\ar[rr]^{v}\ar[dr]_{j\times\P^{1}} &  & M\ar[dl]_{\pi}\ar[ddl]^{q}\\
 & \P^{n}\times\P^{1}\ar[d]\\
 & \P^{1}
}
\]
commute. It is moreover clear that over $P\neq0\in\P^{1}$ the map
$v_{P}$ gives the embedding $X\times\left\{ P\right\} \to\P^{n}\times\left\{ P\right\} $
induced by $j$, while $v_{0}$ is the embedding 
\[
v_0: X \simeq X\times\left\{ 0\right\} \to E \simeq\P_{X}\left(\mathcal{N}_{X|\P^{n}}\op\mathcal{O}_{X}\right)
\]
considered in Lemma \ref{lem:specialcaseprojectivebundle}. Denote
with $p:X\times\P^{1}\to X$ the projection, then, since $M$ is smooth
projective, there exists a locally free resolution by coherent $\mathcal{O}_{M}$-modules
\[
0\to\F_{r}\to\dots\to\F_{0}\to v\da p\ua\E\to0
\]
(note that $v\da p\ua\E$ is coherent). Let $\infty\in\P^{1}\backslash\left\{ 0\right\} $
be another point, then $\left[0\right]\sim\left[\infty\right]$ are
linearly equivalent in $\P^{1}$, therefore
\[
\sum_{k=0}^{r}\left(-1\right)^{k}\ch\left(\F_{k}\right)\cdot q\ua\left(\left[0\right]-\left[\infty\right]\right)=0
\]
in $A_{\bullet}\left(M\right)\ot\Q$. As we noted, we have
\begin{align*}
q\ua\left[0\right] & =\left[E\right]+\left[Y\right]\\
q\ua\left[\infty\right] & =\left[\P^{n}\right],
\end{align*}
where $[E]$, $[Y]$, $[\P^n]$ denote the classes of the images of $i_E$, $i_Y$, $i_\infty$ in $A_\bullet(M)$.
This gives the relation in $A_\bullet(M)\otimes\Q$:
\[
i_{Y*}\left(\sum_{k=0}^{r}\left(-1\right)^{k}\ch\left(\F_{k|Y}\right)\right)+i_{E*}\left(\sum_{k=0}^{r}\left(-1\right)^{k}\ch\left(\F_{k|E}\right)\right)=i_{\infty*}\left(\sum_{k=0}^{r}\left(-1\right)^{k}\ch\left(\F_{k|\P^{n}}\right)\right).
\]

Note that the sheaf $v\da p\ua\E$ is zero on $Y$, therefore 
\[
0\to\F_{r|Y}\to\dots\to\F_{0|Y}\to0
\]
is exact and hence $i_{Y*}\left(\sum_{k=0}^{r}\left(-1\right)^{k}\ch\left(\F_{k|Y}\right)\right)=0$. The term for $E$ is $i_{E*} \left( \ch( v_{0!} (\E) ) \right)$, since $\F_{\bullet|E}$ is a resolution of $v_{0*} \E$. By Lemma \ref{lem:specialcaseprojectivebundle}, this is $i_{E*} \left( v_{0*} \left( \dfrac{\ch(\E)}{\td(\mathcal{N}_{X|E})} \right) \right)$, where $\mathcal{N}_{X|E} \simeq \mathcal{N}_{X|\P^n}$. The term for $\P^n$ is $i_{\infty*} \left( \ch( j_! (\E) ) \right)$, since $\F_{\bullet|\P^n}$ is a resolution of $v_{\infty*} \E \simeq j_* \E$. 

The relation in $A_\bullet(M)\otimes\Q$ thus simplifies to:
\[
(i_E \circ v_0)_* \left( \frac{\ch\left(\E\right)}{\td\left(\mathcal{N}_{X|\P^{n}}\right)} \right) = i_{\infty*} \left( \ch(j_! \E) \right)
\]

Let $g: M \to \P^n$ be the projection map (il blow-down). We apply the pushforward $g_*$ to both sides of the equation.
By functoriality of the pushforward (Lemma \ref{lem:pushforwardfunctorial}) and noting that $g \circ i_E \circ v_0 = j$ (the original embedding) and $g \circ i_\infty = \mathrm{id}_{\P^n}$, we get:
\[
j_* \left( \frac{\ch\left(\E\right)}{\td\left(\mathcal{N}_{X|\P^{n}}\right)} \right) = \ch(j_! \E)
\]
which, by Proposition \ref{prop:cherntoddadditive} applied to the resolution $\F_{\bullet|\P^n}$, is exactly the relation
\[
j_* \left( \frac{\ch\left(\E\right)}{\td\left(\mathcal{N}_{X|\P^{n}}\right)} \right) = \sum_{k=0}^{r}\left(-1\right)^{k}\ch\left(\F_{k|\P^{n}}\right)
\]
required in the hypothesis of Lemma \ref{lem:enoughtoverifythis}.
\end{proof}
\newpage
\subsection{Some consequences}

\subsubsection{Explicit formulas in low dimension}
\begin{thm}[Riemann-Roch for curves]
Let $\mathcal{E}$ be locally free sheaf on $C$ of rank r, then
$$\chi(C, \mathcal{E}) = \deg(\mathcal{E}) + r \cdot (1 - g)$$
\end{thm}

\begin{proof}
In view of Theorem \ref{thm:HRR} we need to see that
\[ \int_C \mathrm{ch}(\mathcal{E}) \cdot \mathrm{td}(\mathrm{T}_C) = \deg(\mathcal{E}) + r \cdot (1 - g) \]
We have $A_\bullet(C) = A_1(C) \oplus A_0(C)$ with $A_0(C)\cdot A_0(C) =0$. In view of Definition \ref{defn:ChernToddcharacters} we have
\[ \mathrm{td}(\mathrm{T}_C) = 1 + \frac{1}{2}c_1(\mathrm{T}_C) , \]
Moreover since $\mathrm{T}_C = (\omega_C)^\lor$, we see that $\deg(c_1(\mathrm{T}_C)) = - \deg(\omega_C) =  2-2g$.
Again in view of Definition \ref{defn:ChernToddcharacters} $\mathrm{ch}(\mathcal{E}) = \mathrm{rank}(\mathcal{E}) + c_1(\mathcal{E}) = r + c_1(\mathcal{E})$, whence
\[\mathrm{ch}(\mathcal{E}) \cdot \mathrm{td}(\mathrm{T}_C) = (r + c_1(\mathcal{E})) \cdot \left(1 + \frac{1}{2}c_1(\mathrm{T}_C)\right) =r+c_1(\mathcal{E}) + \frac{r}{2}c_1(\mathrm{T}_C)\]
since $c_1(\mathcal{E}) \cdot \dfrac{1}{2}c_1(\mathrm{T}_C)\in A_{-1}(C)=0$. Finally we compute
\begin{align*}
\int_C \mathrm{ch}(\mathcal{E}) \cdot \mathrm{td}(\mathrm{T}_C) &= \deg\left( c_1(\mathcal{E}) + \frac{r}{2}c_1(\mathrm{T}_C) \right) \\
&= \deg(c_1(\mathcal{E})) + \frac{r}{2} \deg(c_1(\mathrm{T}_C)) \\
&= \deg(\mathcal{E}) + r(1-g).
\end{align*}
\end{proof}

\begin{thm}[Riemann-Roch for surfaces]
Let Let $\mathcal{L}\in\mathrm{Pic}\left(S\right)$ be locally free
sheaf on $S$ of rank 1, then
\[
\chi\left(S,\mathcal{L}\right)=\frac{1}{2}\mathcal{L}\cdot\left(\mathcal{L}\ot_{\mathcal{O}_{S}}\om_{S}^{\lor}\right)+\chi\left(S,\mathcal{O}_{S}\right),
\]
moreover
\[
\chi\left(S,\mathcal{O}_{S}\right)=\frac{1}{12}\left(K_{S}^{2}+\int_{S}c_{2}\left(\mathrm{T}_{S}\right)\right).
\]
\end{thm}

\begin{proof}
Let $D$ be a divisor such that $\mathcal{L}\simeq\mathcal{O}_{S}\left(D\right)$,
then we compute
\begin{align*}
\int_{S}\ch\left(\mathcal{L}\right)\cdot\td\left(S\right) & =\int_{S}\left(1+c_{1}\left(\mathcal{L}\right)+\frac{1}{2}c_{1}\left(\mathcal{L}\right)^{2}\right)\cdot\left(1+\frac{1}{2}c_{1}\left(\mathrm{T}_{S}\right)+\frac{1}{12}\left(c_{1}\left(\mathrm{T}_{S}\right)^{2}+c_{2}\left(\mathrm{T}_{S}\right)\right)\right)\\
 & =\int_{S}\left(1+\frac{1}{2}c_{1}\left(\mathcal{L}\right)^{2}+\frac{1}{2}c_{1}\left(\mathcal{L}\right)\cdot c_{1}\left(\mathrm{T}_{S}\right)+\frac{1}{12}\left(c_{1}\left(\mathrm{T}_{S}\right)^{2}+c_{2}\left(\mathrm{T}_{S}\right)\right)\right)\\
 & =\frac{1}{2}D^{2}-\frac{1}{2}D\cdot K_{S}+\frac{1}{12}\left(K_{S}^{2}+\int_{S}c_{2}\left(\mathrm{T}_{S}\right)\right).
\end{align*}
The last statement follows setting $\mathcal{L}=\mathcal{O}_{S}$.
\end{proof}
\newpage
\appendix
\section{Koszul complex}
\begin{defn}
Let $A$ be a commutative ring and $E$ a finitely generated $A$-module.
For an element 
\[
s\in E^{\lor}=\hom_{A}\left(E,A\right)
\]
we define $\bigwedge^{\bullet}\left(s\right)$ as the complex\footnote{\begin{xca*}
Check that this is a complex
\end{xca*}
}
\[
\dots\to\bigwedge^{n}E\overset{d_{n}}{\longrightarrow}\bigwedge^{n-1}E\to\dots\to E\overset{s}{\longrightarrow}A\to0
\]
where 
\[
d_{n}\left(e_{1}\wedge\dots\wedge e_{n}\right)=\sum_{i=1}^{n}\left(-1\right)^{i+1}s\left(e_{i}\right)e_{1}\wedge\dots\wedge\hat{e_{i}}\wedge\dots\wedge e_{n}.
\]
We call this complex the Koszul complex associated with $s$. We denote
its homology with
\[
H_{k}\left(s\right)=H_{k}\left(\bigwedge^{\bullet}\left(s\right)\right)=\frac{\ker\left(d_{k}\right)}{\mathrm{Im}\left(d_{k+1}\right)}.
\]
We say that $s\in E^{\lor}$ is regular if $H_{k}\left(s\right)=0$
for $k\neq0$.
\end{defn}

\begin{example}
\label{exa:regularsequencekoszul}Let $A$ be a commutative ring and
let $a_{1},\dots,a_{d}\in A$ with $\left(a_{1},\dots,a_{d}\right)\neq A$.
If $s$ is associated to that sequence, we have
\[
H_{0}\left(s\right)=\frac{A}{\left(a_{1},\dots,a_{s}\right)}.
\]
\end{example}

\begin{rem}
\label{rem:koszultor}Let $\bigwedge^{\bullet}\left(s\right)$ be the Koszul
complex associated with the sequence $a_{1},\dots,a_{d}$ and let
$I=\left(a_{1},\dots,a_{d}\right)$, then all the differentials in
the complex $\bigwedge^{\bullet}\left(s\right)\ot_{A}\frac{A}{I}$ vanish.
Therefore, when $\bigwedge^{\bullet}\left(s\right)$ is exact (cfr. for
example Proposition \ref{prop:propertieskoszul}), the complex $\bigwedge^{\bullet}\left(s\right)\ot_{A}\frac{A}{I}$
computes $\mathrm{Tor}_{\bullet}^{A}\left(\frac{A}{I},\frac{A}{I}\right)$.
\end{rem}

\begin{note}
Recall that, given a commutative ring $A$, a sequence of elements
$a_{1},\dots,a_{d}$ form a regular sequence if $\left(a_{1},\dots,a_{d}\right)\neq A$
and the image of $a_{i}$ is a non-zero-divisor in 
\[
\frac{A}{\left(a_{1},\dots,a_{i-1}\right)}.
\]
\end{note}

\begin{rem}
\label{rem:playingkoszul}Let $a_{1},\dots,a_{d}\in A$ with $I=\left(a_{1},\dots,a_{d}\right)$
and set
\begin{align*}
s:E=A^{d} & \to A\\
\sum c_{i}e_{i} & \mapsto\sum c_{i}a_{i}.
\end{align*}
Let now $F\su E$ the free submodule generated by the first $d-1$
elements of the sequence and let $t=s_{|F}:F\to A$. Denote with $j:F\to E$
the inclusion, then it is split (since $\mathrm{coker}\left(j\right)\simeq A$)
say by $g$, therefore we have a split injective map
\begin{align*}
\wedge^{k}\left(j\right):\bigwedge^{k}F & \to\bigwedge^{k}E\\
f_{1}\wedge\dots\wedge f_{k} & \mapsto j\left(f_{1}\right)\wedge\dots\wedge j\left(f_{k}\right)
\end{align*}
giving a (split) monomorphism
\[
0\to\bigwedge^{\bullet}\left(t\right)\to\bigwedge^{\bullet}\left(s\right),
\]
moreover the map 
\[
\bullet\wedge a_{d}:\bigwedge^{k-1}F\to\ker\left(\bigwedge^{k}E\overset{\wedge^{k}\left(g\right)}{\longrightarrow}\bigwedge^{k}F\right)
\]
can be seen to be an isomorphism and to commute with differentials,
thus giving an exact sequence
\begin{equation}
0\to\bigwedge^{\bullet}\left(t\right)\to\bigwedge^{\bullet}\left(s\right)\to\left(\bigwedge^{\bullet}\left(t\right)\right)\left[-1\right]\to0.\label{eq:exactsequencekoszul-1}
\end{equation}
Taking homology we get an exact sequence
\[
H_{d}\left(t\right)\to H_{d}\left(s\right)\to H_{d-1}\left(t\right)
\]
where $H_{d}\left(t\right)=0$ for rank reasons, hence we have an
exact sequence
\begin{equation}
0\to H_{d}\left(s\right)\to H_{d-1}\left(t\right)\overset{\delta_{d}}{\longrightarrow}H_{d-1}\left(t\right).\label{eq:exactsequencehomologyending}
\end{equation}
Tracking back the construction of the connecting homomorphism of the
Snake Lemma one can see that
\begin{align*}
\delta_{n}:H_{n-1}\left(t\right) & \to H_{n-1}\left(t\right)\\
x & \mapsto\pm a_{d}\cdot x,
\end{align*}
hence in particular $a_{d}\cdot H_{d}\left(t\right)=0$. 
\end{rem}

The connection between the Koszul complex and the properties of the
sequence defining it starts with the following\footnote{Actually the interplay between the two concepts is far tighter: see
for example \cite[Chapter IV]{Ser}, especially Sections A and B.
is done in Example \ref{exa:regularsequencekoszul}.}
\begin{prop}
\label{prop:propertieskoszul}Let $A$ be a commutative ring and let
$a_{1},\dots,a_{d}\in A$ with $I=\left(a_{1},\dots,a_{d}\right)$.
Set
\begin{align*}
s:E=A^{d} & \to A\\
\sum c_{i}e_{i} & \mapsto\sum c_{i}a_{i}
\end{align*}
and denote with $\bigwedge^{\bullet}\left(s\right)^{+}$ the complex
\[
0\to A=\bigwedge^{d}E\to\bigwedge^{d-1}E\to\dots\to E\overset{s}{\longrightarrow}A\to\frac{A}{I}\to0.
\]
Then
\begin{enumerate}
\item If $a_{1},\dots,a_{d}$ is a regular sequence, then $\bigwedge^{\bullet}\left(s\right)^{+}$
is a free resolution of $A/I$.
\item Suppose $A$ is local, noetherian, that $a_{1},\dots,a_{d}\in\m$
and that $H_{1}\left(s\right)=0$. Then $a_{1},\dots,a_{d}$ is a
regular sequence and $\bigwedge^{\bullet}\left(s\right)^{+}$ is exact.
\end{enumerate}
\end{prop}

\begin{proof}
\begin{description}
\item [{1)}] Consider $E=A^{d}$ and 
\begin{align*}
s:E & \to A\\
\sum c_{i}e_{i} & \mapsto\sum c_{i}a_{i}.
\end{align*}
We clearly have $H_{0}\left(s\right)=\frac{A}{I}$. We want to see,
by induction on $d$, that this implies that $s$ is regular: the
case $d=1$ is clear, since in that case the complex is given by
\[
0\to A\overset{\cdot a_{1}}{\longrightarrow}A\to0
\]
with 
\[
H_{1}\left(s\right)=\ker\left(\cdot a\right)=0
\]
by regularity. For greater $d$'s, looking at the exact sequence (\ref{eq:exactsequencehomologyending}),
we have $H_{d-1}\left(t\right)=0$ by induction assumption and hence
also $H_{d}\left(s\right)=0$. Therefore we've seen that a regular
sequence always gives rise to a regular element. 
\item [{2)}] Again this is clear when $d=1$, so we proceed by induction
on $d$. From the exact sequence (\ref{eq:exactsequencekoszul-1})
we get exact
\[
H_{1}\left(s\right)\to H_{1}\left(\left(\bigwedge^{\bullet}\left(t\right)\right)\left[-1\right]\right)=H_{0}\left(t\right)\overset{\delta_{1}}{\longrightarrow}H_{0}\left(t\right)
\]
where $\delta_{1}$ is just the product by $\pm a_{d}$ on $A/\left(a_{1},\dots,a_{d-1}\right)$,
therefore ur assumption that $H_{1}\left(s\right)=0$ gives that $a_{d}$
is a regular element in $A/\left(a_{1},\dots,a_{d-1}\right)$. Again
by induction hypothesis, we only need to see that $H_{1}\left(t\right)=0$,
but the connecting map
\[
H_{2}\left(\left(\bigwedge^{\bullet}\left(t\right)\right)\left[-1\right]\right)=H_{1}\left(t\right)\overset{\delta_{2}}{\longrightarrow}H_{1}\left(t\right)
\]
is just the product by $\pm a_{d}$ and $H_{1}\left(s\right)=0$ forces
it to be surjective, but then we have $a_{d}\in\m$ with $a_{d}\cdot H_{1}\left(t\right)=H_{1}\left(t\right)$.
By noetherianity $H_{1}\left(t\right)$ is finitely generated and
hence it vanishes by Nakayama's lemma. 
\end{description}
\end{proof}

\begin{lem}
\label{lem:koszulannihilatedideal}Let $A$ be a commutative ring
and let $a_{1},\dots,a_{d}\in A$ with $I=\left(a_{1},\dots,a_{d}\right)$.
set
\begin{align*}
s:E=A^{d} & \to A\\
\sum c_{i}e_{i} & \mapsto\sum c_{i}a_{i}.
\end{align*}
Then $I\cdot H_{i}\left(s\right)=0$ for every $i>0$.
\end{lem}

\begin{proof}
We proceed by induction on $d$, the case $d=1$ being clear. For
$d>1$, in view of Remark \ref{rem:playingkoszul} we have exact sequences
\[
0\to\frac{H_{j}\left(t\right)}{a_{d}\cdot H_{j}\left(t\right)}\to H_{j}\left(s\right)\to\mathrm{Ann}_{H_{j-1}\left(t\right)}\left(a_{d}\right)\to0
\]
and we conclude since both edges are annihilated by $I$.
\end{proof}
Of course everything may be globalized by patching
\begin{defn}
Given a projective variety $X$ and a locally free coherent $\mathcal{O}_{X}$-module
$\E$ of rank $r$, for any element $s\in H^{0}\left(X,\E\right)$,
denoting with $\mathcal{I}\su\mathcal{O}_{X}$ the image of $s^{\lor}:\E^{\lor}\to\mathcal{O}_{X}$,
we can associate a complex
\[
\bigwedge^{\bullet}\left(s\right):\quad0\to\bigwedge^{r}\E^{\lor}\to\dots\to\E^{\lor}\to\mathcal{O}_{X}\to\frac{\mathcal{O}_{X}}{\mathcal{I}}\to0.
\]
\end{defn}

\begin{prop}
\label{prop:zerolocusimage}Given a smooth projective variety $X$
and a locally free coherent $\mathcal{O}_{X}$-module $\E$ of rank $r$, for
any element $s\in H^{0}\left(X,\E\right)$ the homology 
\[
H_{i}\left(s\right)=\mathcal{H}_{i}\left(\bigwedge^{\bullet}\left(s\right)\right)
\]
has support on the zero locus $Z\left(\mathrm{Im}\left(s^{\lor}\right)\right)$.
\end{prop}

\begin{proof}
Everything boils down to showing that, if a sequence $a_{1},\dots,a_{d}$
in a ring $A$ gives $\left(a_{1},\dots,a_{d}\right)=A$, then the
corresponding Koszul complex is exact, but this follows from Lemma
\ref{lem:koszulannihilatedideal}.
\end{proof}
\newpage

\nocite{*}
\bibliographystyle{cas-model2-names}
\bibliography{referencesRiemannRoch}

\end{document}